\documentclass[a4paper]{amsproc}
\usepackage{amssymb,xy,color}
\xyoption{all}
\usepackage{ytableau,hyperref}
\usepackage{youngtab}
\usepackage{amsmath}
\usepackage{tikz}
\usepackage[margin= 0.9 in]{geometry}
\numberwithin{equation}{section}
\usepackage[parfill]{parskip}
\hypersetup{
	colorlinks   = true, 
	urlcolor     = blue, 
	linkcolor    = magenta, 
	citecolor   = blue 
}

\makeatletter
\@namedef{subjclassname@2020}{%
	\textup{2020} Mathematics Subject Classification}
\makeatother
\usepackage{subcaption}
\definecolor{darkred}{rgb}{0.7,0,0} 
\newcommand{\defn}[1]{{\color{darkred}\emph{#1}}} 

\newcommand{\sgn}{\mathrm{sgn}}
\newcommand{\MP}{\mathcal{MP}}
\newcommand{\Pa}{\mathcal{P}}
\newcommand{\CC}{\mathbf C}

\newcommand{\ZZ}{\mathbf Z}
\newcommand{\spn}{\mathrm{span}}
\newcommand{\Sym}{\mathrm{Sym}}
\DeclareMathOperator{\Supp}{Supp}

\DeclareMathOperator{\End}{End}

\newcommand{\rn}{\mathrm{rank}}
\DeclareMathOperator{\Hom}{Hom}
\DeclareMathOperator{\set}{set}

\DeclareMathOperator{\Res}{Res}

\newtheorem{theorem}{Theorem}[section]
\newtheorem{lemma}[theorem]{Lemma}
\newtheorem{proposition}[theorem]{Proposition}
\newtheorem{corollary}[theorem]{Corollary}

\theoremstyle{definition}
\newtheorem{definition}[theorem]{Definition}
\newtheorem{remark}[theorem]{Remark}
\newtheorem{example}[theorem]{Example}
\title{The Multiset Partition Algebra}
\author[Narayanan]{Sridhar Narayanan}
\address{Department of Mathematics, IIT Bombay, India}\email{sridharp.narayanan@gmail.com}
\author[Paul]{Digjoy Paul}
\address{Tata Institute of Fundamental Research, Mumbai, India}\email{digjoypaul@gmail.com}
\author[Srivastava]{Shraddha Srivastava}
\address{Uppsala University, Sweden}\email{maths.shraddha@gmail.com}

\begin{document}
	
	\begin{abstract}
		We introduce the multiset partition algebra $\MP_k(\xi)$ over the polynomial ring $F[\xi]$, where $F$ is a field of characteristic $0$ and $k$ is a positive integer.  When $\xi$ is specialized to a positive integer $n$, we establish the Schur--Weyl duality between the actions of
		resulting algebra $\MP_k(n)$ and the symmetric group $S_n$ on $\Sym^k(F^n)$. The construction of $\MP_k(\xi)$ generalizes to any vector $\lambda$ of non-negative integers yielding the algebra $\MP_{\lambda}(\xi)$ over $F[\xi]$ so that there is Schur--Weyl duality between the actions of $\MP_{\lambda}(n)$ and $S_n$ on $\Sym^{\lambda}(F^n)$.
		We find the generating function for the multiplicity  of each irreducible representation of $S_n$  in $\Sym^\lambda(F^n)$, as $\lambda$ varies, in terms of a plethysm of Schur functions. As  consequences we obtain an indexing set for the irreducible representations of $\MP_k(n)$ and
		the generating function for the multiplicity of an irreducible polynomial representation of $GL_n(F)$ when restricted to $S_n$.  We show that $\MP_\lambda(\xi)$ embeds inside the partition algebra $\Pa_{|\lambda|}(\xi)$. Using this embedding,  we show that the multiset partition algebras are generically semisimple over $F$. Also, for the specialization of $\xi$ at $v$ in $F$, we prove that $\MP_{\lambda}(v)$ is a cellular algebra.
	\end{abstract}
	\maketitle
	\tableofcontents
	
	\section{Introduction}
	The symmetric group $S_k$ acts on the $k$-fold tensor space $(\CC^n)^{\otimes k}$ by permuting its tensor factors. The space $\CC^n$ is the defining representation of the general linear group $GL_n(\CC)$, and $GL_n(\CC)$ acts on $(\CC^n)^{\otimes k}$ diagonally. These two actions commute and  they generate the centralizers of each other. Further, as a $(GL_n(\CC),S_k)$-bimodule 
	\begin{displaymath}
		(\CC^n)^{\otimes k}\cong\bigoplus_{\mu} W_{\mu}\otimes V_{\mu},
	\end{displaymath}   
	where $\mu$ is a partition of $k$ with at most $n$ parts, $W_{\mu}$ is an irreducible polynomial representation of $GL_n(\CC)$ and $V_{\mu}$ is a Specht module of $S_n$. 
	This phenomenon, discovered by Schur~\cite{Schur1927} and later popularized by Weyl~\cite{Weyl}, is called the classical Schur--Weyl duality and is a cornerstone of representation theory.
	
	Brauer~\cite{MR1503378} restricted the action of $GL_n(\CC)$ on $(\CC^n)^{\otimes k}$ to the orthogonal group $O_n(\CC)$ and showed Schur--Weyl duality  between the actions of Brauer algebra and $O_n(\CC)$ on $(\CC^n)^{\otimes k}$. Jones~\cite{Jones} and Martin~\cite{MR1265453}, independently, further restricted the action of $O_n(\CC)$ to the symmetric group $S_n$. This led to the definition of partition algebra $\Pa_{k}(\xi)$  over 
	the polynomial ring $F[\xi]$, where $F$ is a field of characteristic $0$. 
	The algebra $\Pa_k(\xi)$ has a basis given by the partitions of a set of cardinality $2k$. A partition of such a set can be pictorially represented as a graph (see Section \ref{pa}), called the partition diagram. In particular, $\Pa_{k}(\xi)$ is a diagram algebra as described by Martin~\cite{martin2008diagram}. Martin and Saluer~\cite{MS94} showed that $\Pa_k(\xi)$ is semisimple over $F[\xi]$. Further, 
	for the specialization of $\xi$ at $v\in F$, $\Pa_k(v)$ is semisimple over $F$ unless $v$ is an integer lying in the set $\{0,1,\ldots, 2k-2\}$. 
	
When $\xi$ is evaluated at a positive integer $n$,	Jones~\cite{Jones} showed that   $\Pa_k(n)$ maps onto the commutant of $S_n$ acting on $(F^n)^{\otimes k}$.
	Moreover, when $n\geqslant 2k$,
	\begin{equation}\label{eq:intro_1}
		\Pa_k(n)\cong \End_{S_n}((F^n)^{\otimes k}).
	\end{equation}
	In particular, this isomorphism gives a diagrammatic interpretation of the centralizer algebra $\End_{S_n}((F^n)^{\otimes k})$.
	
	Martin~\cite{MR1265453} classified the irreducible representations of $\Pa_k(\xi)$. Halverson and Ram~\cite{HR} gave a combinatorial description of these representations.  For $n\geqslant 2k$,  Benkart and Halverson~\cite{BH} constructed a basis for the irreducible representations of $\Pa_k(n)$ in terms of the set partition tableaux. Orellana and Zabrocki~\cite{OZ} studied the combinatorics of the multiset tableaux, which generalizes the set partition tableaux. The multiset tableaux play a central role in the representation theory of the algebra we define in this paper.

	Inspired by the Kazdhan--Lusztig basis of Hecke algebras, Graham and Lehrer~\cite{GL} introduced the notion of a cellular algebra in terms of the existence of a special kind of basis, called  cellular basis. An important application of cellular basis is that it allows one to construct the irreducible representations of the cellular algebra. K\"onig and Xi~\cite{KX} gave a basis-free definition of a cellular algebra. Prominent examples of cellular algebras include Ariki--Koike Hecke algebras, Brauer algebras, and also partition algebras.

	The $k$-th symmetric power $\Sym^{k}(F^n)$ is a polynomial representation of $GL_n(F)$.  In spirit of the isomorphism~\eqref{eq:intro_1}, we 
	study the centralizer algebra $\End_{S_n}(\Sym^k(F^n))$ by constructing a new diagram algebra $\MP_k(\xi)$ over $F[\xi]$ in Section~\ref{sec:mpa}. This algebra has a basis indexed by the multiset partitions of the multiset $\{1^k,1'^k\}$. Pictorially these multiset partitions can be represented by certain multigraphs (Equation~\eqref{al:bij}). The structure constants of $\MP_k(\xi)$ with respect to this basis are polynomials in $\xi$ (Equation \eqref{eq:poly_gene}).

	The Schur--Weyl duality between the actions of $\MP_k(n)$ and $S_n$ on $\Sym^{k}(F^n)$ is demonstrated in  Theorem~\ref{thm:swd_mult}.  This provides a diagrammatic interpretation of $\End_{S_n}(\Sym^k(F^n))$, thereby answering a question in~\cite[p. 22]{Nate}. In Theorem~\ref{thm:mult_k}, we conclude that $\MP_k(\xi)$ is an unital associative algebra. Following Harman \cite[p. 22]{Nate}, we call $\MP_{k}(\xi)$ the multiset partition algebra.
	
	For a vector $\lambda=(\lambda_1,\ldots, \lambda_s)$ of non-negative integers, the representation 
	$$\Sym^{\lambda}(F^n):=\Sym^{\lambda_1}(F^n)\otimes\cdots\otimes\Sym^{\lambda_s}(F^n)$$
	along with the tensor product of exterior powers of $F^n$ are the building blocks for constructing the irreducible polynomial representations of $GL_n(F)$. By the restriction $\Sym^{\lambda}(F^n)$
	is also a representation of $S_n$. For a partition $\nu$ of $n$, let $a_{\nu}^{\lambda}$ denote the multiplicity of the irreducible representation of $S_n$ corresponding to $\nu$ in $\Sym^{\lambda}(F^n)$. These multiplicities are interpreted as a count of certain multiset tableaux in~\cite{OZ} and in~\cite[Proposition 3.11]{Nate}. In Lemma~\ref{lm:gen_lambda} we obtain the ordinary generating function for $\{a^{\lambda}_{\nu}\}_{\lambda}$. As a result, we give an indexing set for the irreducible representations of $S_n$ occurring in $\Sym^{\lambda}(F^n)$ when $\lambda=(k)$ in Theorem~\ref{coro:mult}. Due to Schur--Weyl duality between $\MP_k(n)$ and $S_n$, this is also an indexing set of the irreducible representations of $\MP_k(n)$ when $n\geqslant 2k$. Another application  of this generating function yields a generating function for the restriction coefficients (see Proposition~\ref{theorem:r-gen}), i.e., the multiplicities of an irreducible representation of $S_n$ in the restriction of an irreducible polynomial representation of $GL_n(F)$ to $S_n$.

	 We construct an embedding of $\MP_k(\xi)$ in $\Pa_k(\xi)$ in Theorem~\ref{cor:new_lambda} and show that the image of
	the embedding is $e\Pa_k(\xi)e$ for an idempotent $e\in \Pa_k(\xi)$ in Lemma~\ref{lm:idem}.  Upon specialization of $\xi$ at $v\in F$,  in Theorem~\ref{thm:gensem} we obtain that  $\MP_{k}(v)$ is semisimple when $v$ is not an integer or $v$ is an integer such that $v\notin\{0,1,\ldots,2k-2\}$.  In Theorem~\ref{thm:cellular}, we prove $\MP_k(v)$ is cellular.

	In Section~\ref{app:A}, we generalize the construction of $\MP_k(\xi)$ to any vector $\lambda$ of non-negative integers. The algebra so obtained has a basis indexed by the multiset partitions of $\{1^{\lambda_1},\ldots,s^{\lambda_s}, 1'^{\lambda_1},\ldots,s'^{\lambda_s}\}$, which we pictorially represent by certain multigraphs.  
	The Schur--Weyl duality exists between the actions of $\MP_{\lambda}(n)$ and $S_n$ on $\Sym^{\lambda}(F^n)$. In particular, for $n \geqslant 2|\lambda|$,  $\MP_{\lambda}(n)$ is isomorphic to the centralizer algebra $\End_{S_n}(\Sym^{\lambda}(F^n))$. We also call $\MP_\lambda(\xi)$ the multiset partition algebra. 
	We define an embedding of $\MP_\lambda(\xi)$ inside $\Pa_{|\lambda|}(\xi)$. The image of
	the embedding is $e\Pa_{|\lambda|}(\xi)e$ for an idempotent $e\in \Pa_{|\lambda|}(\xi)$. As an application of this embedding, we prove that $\MP_\lambda(\xi)$ is semisimple over $F[\xi]$. Further, upon specialization of $\xi$ at $v\in F$,  we obtain that  $\MP_{\lambda}(v)$ over $F$ is cellular, and semisimple when $v$ is not an integer or $v$ is an integer such that $v\notin\{0,1,\ldots,2|\lambda|-2\}$. The aforementioned embedding is an isomorphism when $\lambda=(1^k)$, so that
	\begin{displaymath}
		\Pa_k(\xi)\cong \MP_{(1^k)}(\xi).
	\end{displaymath}
The proofs of the results in this section follow by minor modifications of arguments in Sections~\ref{sec:mpa} and \ref{sec:five}. We outline these modifications where required.

		Concurrently in \cite{OZM}, Orellana and Zabrocki  studied a diagrammatic algebra, also called a multiset partition algebra. This algebra is related to the centralizer algebra $\End_{S_n}(\Sym^{r}(\CC^{n}\otimes \CC^k))$. The  multiset partition algebra defined in this manuscript is closely related to their algebra. However, we do not give an explicit connection here.
	
	Throughout this paper, $F$ denotes a field of characteristic $0$.
	
	\section{Preliminaries}
	In this section, we give an overview of the multiset tableaux, centralizer algebras of permutation representations of a finite group,  and partition algebras, which are relevant in this paper.

	\subsection{Multisets and tableaux}
	In this section, we define and introduce notation for various terms related to multisets, partitions of multisets, and tableaux. Multiset tableaux occur in the work of Colmenarejo, 
	Orellana, Saliola, Schilling, and Zabrocki \cite{OZS}. Here we include a self-contained recapitulation of these concepts. 
	
	A multiset is a collection of possibly repeated objects. For example, $\{1,1,1,2,2\}$ is a multiset with three occurrences of the element $1$ and two occurrences of the element $2$. More precisely:
	\begin{definition} 
		A \defn{multiset} is an ordered pair $(S,f)$ where $S$ is a set, and $f$ is a non-negative integer-valued function on S, which we call the multiplicity function.
	\end{definition} 
	
	Given $S=\{1,2,\ldots,n\}$ and a multiplicity function $f$, we shall denote 
	\begin{displaymath}
		(S,f)=\{1^{f(1)},2^{f(2)},\ldots, n^{f(n)}\}.
	\end{displaymath}
	Here $f(i)$ indicates the number of times $i$ appears in the multiset. Thus the aforementioned multiset would be denoted $\{1^3,2^2\}$. Given a set $S$, denote the set of all multisets with elements drawn from $S$ by $\mathcal{A}_S$.

	\begin{definition}
		A \defn{multiset partition} of a given multiset $M$ is a multiset of multisets whose disjoint union equals $M$.
	\end{definition}
	
	For example, $\pi= \{\{1,2\},\{1,2\},\{2\},\{1^2\}\}$ is a multiset partition of $\{1^4,2^3\}$.
	
	A set with a total order on its elements is called an ordered set. Given an ordered set $S$, we fix the order on multisets in $\mathcal{A}_S$ to be the \defn{graded lexicographic order} $\leqslant_{L}$, which we recall below from \cite[Section 2.4]{OZS}.
	\begin{definition}\label{def:go}
		For multisets $M_1=\{a_1,\ldots, a_r\}$ with $a_1\leqslant \cdots \leqslant a_r$, and $M_2=\{b_1,\ldots, b_s\}$ with $b_1\leqslant \cdots \leqslant b_s$, we say $M_1 \leqslant_{L} M_2$ in this order if
		\begin{itemize}
			\item  $r \leqslant s$ and
			\item if $r=s$, then there exists a positive integer $i \leqslant r$ such that $a_j=b_j$ for all $1\leqslant j < i$ and $a_i < b_i$. 
		\end{itemize}
	\end{definition}

	\begin{definition}
		\label{def:ssmt}
		Let $\lambda$ be a partition, and $S$ be an ordered set.  A \defn{semistandard multiset tableau (SSMT)} of shape $\lambda$ is a filling of the cells of the Young diagram of $\lambda$ with the entries from $\mathcal{A}_S$ such that
		\begin{itemize}
			\item the entries increase strictly along each column and
			\item the entries increase weakly along each row.
		\end{itemize}
	\end{definition}
	The \defn{content} of a  SSMT $P$ is the multiset obtained by the disjoint union of entries of $P$ (see \cite[Section 2.6]{OZS}). Let SSMT($\lambda,C)$ denote the set of all semistandard multiset tableau of shape $\lambda$ and content $C$.

	\subsection{Centralizer algebras}
	\label{sec:twist-perm-repr}
	For a finite set $X$, let  $F[X]$ denote the space of $F$-valued functions on $X$. The  set of \defn{indicator functions} $1_x$ for $x\in X$ is a basis of $F[X]$. By a $G$-set $X$, we mean a group $G$ acts on $X$ on the left.  Then $G$ also acts on $F[X]$ as follows:
	\begin{equation}
		\label{eq:perm-rep}
		g\cdot f(x) = f(g^{-1}{x}), \text{ for } x\in X, \; g\in G, \text{ and } f\in F[X].
	\end{equation} 
	The space $F[X]$ is called the \defn{permutation representation} of $G$ associated to the $G$-set $X$.
	Suppose that $X$ and $Y$ are finite $G$-sets.
	Given a function $\phi:X\times Y\to F$, the \defn{integral operator} $\zeta_\phi:F[Y]\to F[X]$ associated to $\phi$ is defined as
	\begin{equation}
		\label{eq:integral-operator}
		(\zeta_\phi f)(x) = \sum_{y\in Y} \phi(x,y)f(y), \text{ for } f\in F[Y].
	\end{equation}
	If $Z$ is another $G$-set and $\phi:X\times Y\to F$ and $\psi:Y\times Z \to F$ are functions, then
	\begin{displaymath}
		\zeta_\phi \circ \zeta_{\psi} = \zeta_{\phi*\psi},
	\end{displaymath}
	where $\phi*\psi:X\times Z\to F$ is the {convolution product}
	\begin{equation}
		\label{eq:convolution-product}
		\phi*\psi(x,z) = \sum_{y\in Y} \phi(x,y)\psi(y,z).
	\end{equation}
	
	Let $(G\setminus X\times Y)$ denote the set of orbits of the diagonal action of $G$ on $X\times Y$.
	From~\cite[Theorem 2.4.4]{rtcv} we have:
	\begin{theorem}
		\label{corollary:twisted-intertwiner}
		Let $X$ and $Y$ be finite $G$-sets and for $O\in (G\setminus X\times Y)$, define
		\begin{displaymath}
			\phi_O(x,y) =
			\begin{cases}
				1 &\text{if } (x,y)\in O,\\
				0 &\text{otherwise}.
			\end{cases}
		\end{displaymath}
		Write $\zeta_O=\zeta_{\phi_O}$.
		Then the set
		\begin{displaymath}
			\{\zeta_O\mid O\in (G\setminus X\times Y)\}
		\end{displaymath}
		is a basis for $\Hom_G(F[Y],F[X])$.
		Consequently,
		\begin{displaymath}
			\dim \Hom_G(F[Y],F[X]) = |G\setminus (X\times Y)|.
		\end{displaymath}
	\end{theorem}

	\subsection{Partition algebras}\label{pa}
	Let $k$ be a positive integer. A \defn{set partition} of $\{1,2,\ldots,k,1',2',\ldots,k'\}$
	is of the form $\{B_1,B_2,\ldots,B_l\}$, where 
	$B_1,B_2,\ldots,B_l$ are mutually disjoint sets such that 
	\begin{displaymath}
		\sqcup_{i=1}^l B_i= \{1,2,\ldots,k,1',2',\ldots,k'\}.
	\end{displaymath}
	A set partition can be drawn as a 
	graph, called a \defn{partition diagram}, which has vertices on two rows. The vertices of the top and the bottom rows are $\{1,2,\ldots,k\}$ and $\{1',2',\ldots,k'\}$
	respectively. For $i,j\in\{1,2,\ldots,k,1',2',\ldots,k'\}$, there is a path between $i$ and $j$ if only if 
	$i,j\in B_s$, for $1\leqslant s\leqslant l$. If the part $B_{s}$ is non-empty, it is called  a \defn{block} in the corresponding partition diagram. Unless stated otherwise, we omit empty parts in writing a set partition. Two partition diagrams are considered to be the same  if and only if they have the same underlying set partition.
	\begin{example}
\label{eg:block}
		The partition 
		$$\{\{1,2,1',3'\},\{3,5,4'\},\{4,2',5'\}\}$$ of  $\{1,2,3,4,5,1',2',3',4',5'\}$ corresponds to the following partition diagram:
		
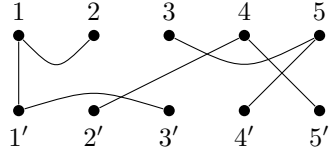
\begin{center}
			\begin{tikzpicture}
				[scale=1,
				mycirc/.style={circle,fill=black, minimum size=0.1mm, inner sep = 1.5pt}]
				
				\node[mycirc,label=above:{$1$}] (n1) at (0,1) {};
				\node[mycirc,label=above:{$2$}] (n2) at (1,1) {};
				\node[mycirc,label=above:{$3$}] (n3) at (2,1) {};
				\node[mycirc,label=above:{$4$}] (n4) at (3,1) {};
				\node[mycirc,label=above:{$5$}] (n5) at (4,1) {};
				\node[mycirc,label=below:{$1'$}] (n1') at (0,0) {};
				\node[mycirc,label=below:{$2'$}] (n2') at (1,0) {};
				\node[mycirc,label=below:{$3'$}] (n3') at (2,0) {}; 
				\node[mycirc,label=below:{$4'$}] (n4') at (3,0) {}; 
				\node[mycirc,label=below:{$5'$}] (n5') at (4,0) {}; 
				
				\draw (n1)--(n1');
				\draw (n1)..controls(0.5,0.5).. (n2);
				\draw (n1')..controls(1,0.3)..(n3');
				\draw (n4)--(n2');
				\draw (n4)--(n5');
				\draw (n5)--(n4');
				\draw (n3)..controls(3,0.5)..(n5);
			\end{tikzpicture}
		\end{center}
	\end{example}
	Let $\mathcal{A}_k$ denote the set of all partition diagrams on $\{1,2,\ldots,k,1',2',\ldots,k'\}$.
	 The \defn{partition algebra} $\Pa_k(\xi)$ is the free module over the polynomial ring $F[\xi]$ with basis $\mathcal{A}_k$.
	Given $d_1,d_2\in \mathcal{A}_k$, let $d_1\circ d_2$ denote the \defn{concatenation} of $d_1$ with $d_2$, i.e., place $d_2$ on the top of $d_{1}$, and then identify the bottom vertices of $d_2$ with the top vertices of $d_1$ and remove any connected components that lie
	completely in the middle row. Borrowing analogy from the composition of maps, in the composition $d_1\circ d_2$, we apply $d_2$ followed by $d_1$, reading maps from top to bottom. This differs from the standard convention but it has been adapted, for example, in Bloss~\cite[p. 694]{Bloss1}.  Define the \defn{multiplication of partition diagrams} as follows:
	\begin{equation}\label{mult_partition} 
		d_1d_2:=\xi^{[d_1\circ d_2]}d_1\circ d_{2}
	\end{equation}
	where $ [d_1\circ d_2]$ denotes the number of connected components that lie entirely in the middle row while computing $d_1\circ d_{2}$. By linearly extending the multiplication~\eqref{mult_partition}, $\Pa_{k}(\xi)$ becomes an associative unital algebra over $F[\xi]$.
	
	\begin{example}
		The concatenation $d_1\circ d_2$ of $$d_1=\{\{1,1'\},\{2\},\{3\},\{4\},\{2',3'\},\{5,4',5'\}\}$$ with
		$$d_2=\{\{1,2,1'\},\{3,5\},\{2',3'\},\{4'\},\{4,5'\}\}$$ is illustrated below:
		
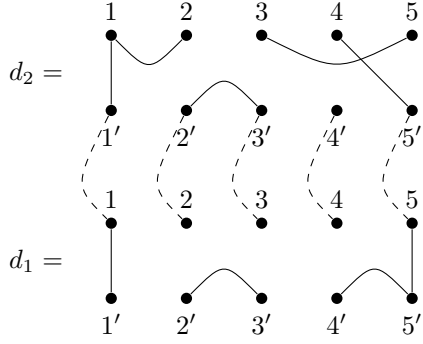
\begin{center}
			\begin{tikzpicture}
				[scale=1,
				mycirc/.style={circle,fill=black, minimum size=0.1mm, inner sep = 1.5pt}]
				\node (1) at (-1,0.5) {$d_2 =$};
				\node[mycirc,label=above:{$1$}] (n1) at (0,1) {};
				\node[mycirc,label=above:{$2$}] (n2) at (1,1) {};
				\node[mycirc,label=above:{$3$}] (n3) at (2,1) {};
				\node[mycirc,label=above:{$4$}] (n4) at (3,1) {};
				\node[mycirc,label=above:{$5$}] (n5) at (4,1) {};
				\node[mycirc,label=below:{$1'$}] (n1') at (0,0) {};
				\node[mycirc,label=below:{$2'$}] (n2') at (1,0) {};
				\node[mycirc,label=below:{$3'$}] (n3') at (2,0) {}; 
				\node[mycirc,label=below:{$4'$}] (n4') at (3,0) {}; 
				\node[mycirc,label=below:{$5'$}] (n5') at (4,0) {}; 
				
				\draw (n1)--(n1');
				\draw (n1)..controls(0.5,0.5).. (n2);
				\draw (n4)--(n5');
				\draw (n3)..controls(3,0.5)..(n5);
				\draw (n2')..controls(1.5,0.5)..(n3');
				
				\node (1) at (-1,-2) {$d_1 =$};
				\node[mycirc,label=above:{$1$}] (n6) at (0,-1.5) {};
				\node[mycirc,label=above:{$2$}] (n7) at (1,-1.5) {};
				\node[mycirc,label=above:{$3$}] (n8) at (2,-1.5) {};
				\node[mycirc,label=above:{$4$}] (n9) at (3,-1.5) {};
				\node[mycirc,label=above:{$5$}] (n10) at (4,-1.5) {};
				\node[mycirc,label=below:{$1'$}] (n6') at (0,-2.5) {};
				\node[mycirc,label=below:{$2'$}] (n7') at (1,-2.5) {};
				\node[mycirc,label=below:{$3'$}] (n8') at (2,-2.5) {}; 
				\node[mycirc,label=below:{$4'$}] (n9') at (3,-2.5) {}; 
				\node[mycirc,label=below:{$5'$}] (n10') at (4,-2.5) {}; 
				
				\draw (n6)--(n6');
				\draw (n7')..controls(1.5,-2)..(n8');
				\draw (n10)--(n10');
				\draw (n9')..controls(3.5,-2)..(n10');
				
				\draw[dashed] (n1')..controls(-0.5,-1)..(n6);
				\draw[dashed] (n2')..controls(0.5,-1)..(n7);
				\draw[dashed] (n3')..controls(1.5,-1)..(n8);
				\draw[dashed] (n4')..controls(2.5,-1)..(n9);
				\draw[dashed] (n5')..controls(3.5,-1)..(n10);
			\end{tikzpicture}
		\end{center}
		In the above, there are exactly two connected components which lie entirely in the middle, so the multiplication
			\begin{center}
			\begin{tikzpicture}
				[scale=1,
				mycirc/.style={circle,fill=black, minimum size=0.1mm, inner sep = 1.5pt}]
				\node (1) at (-1,0.5) {$d_1d_2 =\xi^{2}$};
				\node[mycirc,label=above:{$1$}] (n1) at (0,1) {};
				\node[mycirc,label=above:{$2$}] (n2) at (1,1) {};
				\node[mycirc,label=above:{$3$}] (n3) at (2,1) {};
				\node[mycirc,label=above:{$4$}] (n4) at (3,1) {};
				\node[mycirc,label=above:{$5$}] (n5) at (4,1) {};
				\node[mycirc,label=below:{$1'$}] (n1') at (0,0) {};
				\node[mycirc,label=below:{$2'$}] (n2') at (1,0) {};
				\node[mycirc,label=below:{$3'$}] (n3') at (2,0) {}; 
				\node[mycirc,label=below:{$4'$}] (n4') at (3,0) {}; 
				\node[mycirc,label=below:{$5'$}] (n5') at (4,0) {}; 
				
				\draw (n1)--(n1');
				\draw (n1)..controls(0.5,0.5).. (n2);
				\draw (n4)--(n5');
				\draw (n3)..controls(3,0.5)..(n5);
				\draw (n2')..controls(1.5,0.5)..(n3');
				\draw (n4')..controls(3.3,0.5)..(n5');
			\end{tikzpicture}
		\end{center}
	\end{example}
	Now we define another basis of $\Pa_k(\xi)$, which is particularly essential for this paper.
	Given $d,d'\in \mathcal{A}_k$, we say $d'$ is coarser than $d$, denoted as $d' \leqslant d$, if whenever $i,j\in\{1,2,\ldots,k,1',2',\ldots,k' \}$ are in the same block in $d$, then $i$ and $j$ are in the same block in $d'$.	For a partition diagram $d$, define the element $x_d \in  \Pa_k(\xi)$ by setting:
	\begin{equation}\label{eq:xd}
		d = \sum\limits_{d' \leqslant d} x_{d'}.
	\end{equation}
	It can be easily seen that the transition matrix between $\{d\}$ and  $\{x_d \}$ is unitriangular. Thus $\{x_d \}$ is also a basis, known as the \defn{orbit basis}, of the partition algebra $ \Pa_k(\xi)$.
	
	The structure constants of $\Pa_k(\xi)$ with respect to the orbit basis $\{x_d\}$ is given in~\cite[Theorem 4.8]{BH}. To state this result, we need the following definitions. 
	
	\begin{definition}\label{def:falling_factorial}
		For a non-negative integer $l$ and $f(\xi)\in F[\xi]$, the \defn{falling factorial polynomial} is $(f(\xi))_l=f(\xi)(f(\xi)-1)\cdots(f(\xi)-l+1)$.
	\end{definition}

	\begin{definition}\label{eg:par_alg_diag}
		Given a partition diagram $d=\{B_1,\ldots,B_n\}\in \mathcal{A}_k$, we define 
		\begin{displaymath}
			B_j^u:=B_j \cap \{1,\ldots, k\} \text{ and }  B_j^l:=B_j \cap \{1',\ldots, k'\} \text{ for $1\leqslant j\leqslant n$}.
		\end{displaymath}                                                                                       
		Then we may also represent  $d$ as the set of tuples
		$d=\{(B_1^u, B_1^l),\ldots, (B_n^u, B_n^l)\}$ (this viewpoint is crucially used in Equation~\eqref{eq:diag_operator} and Section \ref{subsec:embed}). Also define $d^u:=\{B_1^u ,\ldots,B_n^u\}$ and $d^l:=\{B_1^l,\ldots,B_n^l\}$.
	\end{definition}

	For $d_1,d_2\in \mathcal{A}_k$, we say $d_1\circ d_2$ \defn{matches} in the middle if the set partition $d_{1}^{u}$ is same as the set partition obtained from $d_2^{l}$ after ignoring the primes on numbers in $d_2^l$. For a partition diagram $d\in \mathcal{A}_k$, let $|d|$ denote the number of blocks in $d$. From~\cite[Lemma 3.1]{MS}, we recall, in the following theorem, the structure constants of $\Pa_{k}(\xi)$ with respect to the basis $\{x_d\}$.

	\begin{theorem}\label{thm:xd} For $d_1,d_2\in \mathcal{A}_k$, we have
		\begin{align}
			x_{d_1}x_{d_2}=\begin{cases}
				\sum_{d} (\xi-|d|)_{[d_1\circ d_2]} x_{d} & \text{ if $d_{1}\circ d_{2}$ matches in the middle},\\
				0 & \text{otherwise},
			\end{cases}
		\end{align}
		where the sum is over all the coarsenings $d$ of $d_1\circ d_2$ which are 
		obtained by connecting a block of $d_2$, which lie entirely in the top row of $d_2$, with a block of $d_1$, which lie entirely in the bottom row of $d_1$. 
		
	\end{theorem}
	\subsection{Schur--Weyl duality}  Let $F^n$ be the $n$-dimensional vector space over $F$ with 
	standard basis $\{e_{1},e_{2},\ldots,e_{n}\}$. The group 
	$S_{n}$ acts on a basis element $e_{i}$ as:
	\begin{displaymath}
		\sigma e_{i}:=e_{\sigma(i)},\text{ where } \sigma\in S_{n}.
	\end{displaymath}
	So $S_{n}$ acts on the $k$-fold tensor product $(F^n)^{\otimes k}$ diagonally.

	For $d\in \mathcal{A}_k$, define
	
	\[(d)^{i_{1},\ldots,i_{k}}_{i_{1'},\ldots,i_{k'}}=\begin{cases}
		1 & \text{ if }
		i_{r}=i_{s}  \text{ iff }r,s\in\{1,\dotsc,k,1',\dotsc,k'\}\\ & \text{ are in the same block of $d$},\\
		0 &\text{ otherwise}.
	\end{cases}\]

	The action of $x_d$ on a basis element of $(F^n)^{\otimes k}$ given by
\begin{equation}\label{eq:part map}
		x_d(e_{i_{1}}\otimes\cdots\otimes e_{i_{k}})=\sum_{i_{1'},\ldots,i_{k'}} (d)^{i_{1},\ldots,i_{k}}_{i_{1'},\ldots,i_{k'}} e_{i_{1'}}\otimes \cdots \otimes e_{i_{k'}}
	\end{equation}
	defines a map $\phi_{k}:\Pa_{k}(n)\rightarrow \End_F((F^n)^{\otimes k})$.
	The Schur--Weyl duality between the actions of the group algebra $F[S_n]$ and the partition algebra $\Pa_k(n)$ on $(F^n)^{\otimes k}$ is given in~\cite{Jones} and in~\cite{Martin}. We state, in the following theorem, the Schur--Weyl duality from~\cite[Theorem 3.6]{HR}, which is more applicable in this paper.
	
	\begin{theorem}\phantomsection \label{thm:SWD_partition}
		\begin{enumerate}
			\item    The image of map $\phi_{k}:\Pa_{k}(n)\rightarrow \End_{F}((F^n)^{\otimes k})$ is $\End_{S_n}((F^n)^{\otimes k})$. The kernel of $\phi_{k}$ is spanned by $$\{x_{d}\mid d\text{ has more than n blocks}\}.$$ In particular, when $n\geqslant 2k$, $\Pa_k(n)$ is isomorphic to $\End_{S_{n}}((F^n)^{\otimes k})$.
			\item  The group  algebra  $F[S_n]$ generates the centralizer algebra $\End_{\Pa_{k}(n)}((F^n)^{\otimes k})$.
		\end{enumerate}    
	\end{theorem}
	
\textbf{Tensor space as a permutation module.} We will interpret $(F^n)^{\otimes k}$ as a permutation module, and describe the action of $x_d$ on this permutation module (Equation \eqref{eq:diag_operator}). This viewpoint will be specifically relevant in the proof of Theorem~\ref{cor:new_lambda}.

	For the rest of this section we fix a positive integer $n$, and fix a basis $e_1,\dotsc,e_n$ of $F^n$. We will deviate from convention specified in Section \ref{pa} by padding a partition diagram containing $s\leq n$ components with $n-s$ empty sets. 
	
	An \defn{ordered set partition} of $\{1,\dotsc,k\}$ into $n$ parts is a sequence ${A}=(A_1,\dotsc,A_n)$ of disjoint sets such that $\sqcup A_i = \{1,\dotsc, k\}$. Let $\mathcal{A}(n,k)$ denote the set of ordered set partitions of $\{1,\dotsc,k\}$ into $n$ parts. The group $S_n$ acts on an element in $\mathcal{A}(n,k)$ by permuting the elements of the tuple. So the space $F[\mathcal{A}(n,k)]$ is a permutation module with basis 
	$\{1_{A}\mid A\in\mathcal{A}(n,k)\}$.

	The tensor space  $(F^n)^{\otimes k}$  has a basis $\{e_{j_1}\otimes \dotsb \otimes e_{j_k}\mid 1\leqslant j_1,\dotsc,j_k \leqslant n \}$. The following map is an $S_n$-linear isomorphism
	\begin{align}
		\label{ord par}
		e_{j_1} \otimes \cdots \otimes e_{j_k} \rightarrow 1_{(A_1,\ldots,A_n)},
	\end{align}
	where $A_m= \{l \mid  e_{j_l}=e_{m},\, 1 \leqslant l \leqslant k\}$  for $1\leqslant m\leqslant n$ and $\sqcup_{m=1}^{n}A_m=\{1,\ldots,k\}$. Note that these sets may be empty.    
	
	\begin{example}
		\label{eg:indicate}
		The element $e_2 \otimes e_1 \otimes e_1$ in $(F^4)^{\otimes 3}$ is represented by the indicator function corresponding to the tuple $(\{2,3\},\{1\},\emptyset,\emptyset)$.
	\end{example}
	We transport  the action \eqref{eq:part map} of $x_d$ on $F[\mathcal{A}(n,k)]$ via the $S_n$-linear isomorphism $(F^n)^{\otimes k}\cong F[\mathcal{A}(n,k)]$.

	For  a tuple $r=(r_1,\ldots,r_n)$, define \defn{\text{$\set(r)$}}:=$\{r_1,\ldots,r_n\}$. We shall use this notation repeatedly, specifically related to the action of $x_d$.
	
	When $d\in \mathcal{A}_k$ has more than $n$ blocks then the map  $x_d$ is the zero map on $F[\mathcal{A}(n,k)]$. Assume that $d\in\mathcal{A}_k$ has at most $n$ blocks so that we can write  $d=\{B_1,\ldots,B_n\}$. 
	Let $A=(A_1,\ldots,A_n)$ and $C=(C_1,\ldots,C_n)$ be ordered set partitions of $\{1,2,\ldots,k\}$ and $\{1',2',\ldots,k'\}$ respectively. Now using Definition \ref{eg:par_alg_diag}, the action of $x_d$ is given  as                                                                                                                                                        
	\begin{align}
		\label{eq:diag_operator}
		x_d(1_A) = \begin{cases}
			\sum_{\{C\mid\{(A_1, C_1),\ldots,(A_n,C_n)\}=d\}}1_C & \text{ if } \,  \set(A)=d^u,\\
			0 & \text{ otherwise}.
		\end{cases}
	\end{align}
	In the above, we make the convention that for an ordered set partition  $C$ of $\{1',2',\ldots,k'\}$,  the element $1_C$ of $F[\mathcal{A}(n,k)]$ is obtained by ignoring the primes on elements of parts of $C$.

	\section{The multiset partition algebra}
	\label{sec:mpa}
	
	For a positive integer $k$, each multiset partition of $\{1^k,1'^k\}$ can be represented by the equivalence class of certain graphs by a procedure described below (see Equation~\eqref{al:bij}). We use these graphs to define the multiset partition algebra $\MP_k(\xi)$. 
	
	Let $\Gamma$ be a \defn{multigraph} (i.e., multiple edges between the same pair of vertices) whose vertices are arranged in two rows. The vertices in the top row are labelled by $0,1,\ldots, k$ and the vertices in the bottom row are also labelled by $0,1,\ldots, k$. Let $E_{\Gamma}$ denote the multiset of edges of $\Gamma$. Every edge of $\Gamma$ connects a vertex in the top row to a vertex in the bottom row. 
	
	We define a \defn{weight function} $w: E_{\Gamma} \to \mathbb{Z}_{\geqslant 0}^2$ by setting $w(e)=(i,j)$ for every edge $e$ from the vertex $i$ in the top row to the vertex $j$ in the bottom row. We say an edge $e$ is non-zero if $w(e)\neq (0,0)$. The \defn{weight of a graph} $\Gamma$ is defined as $w(\Gamma):=\sum_{e\in E_{{\Gamma}}}w(e).$ 
	
	Let $\mathcal{B}_k$ be the set of all such multigraphs $\Gamma$ with the weight $w(\Gamma)=(k,k)$. Two  multigraphs in $\mathcal{B}_k$ are said to be \defn{equivalent} if they have the same non-zero weighted edges.  Denote by $\tilde{\mathcal{B}}_{k}$ the set of all equivalence classes in $\mathcal{B}_{k}$.
	
	The \defn{rank of a graph} $\Gamma$, denoted $ \rn(\Gamma)$, is the number of non-zero weighted edges of $\Gamma$. Since all graphs in the same equivalence class have the same rank, the rank of the equivalence class $[\Gamma]$ is defined as  $\rn([\Gamma])=\rn(\Gamma)$. 
	\begin{example}\label{graph}
		The following two graphs in $\mathcal{B}_5$ have rank four, and they are equivalent. 
		
			
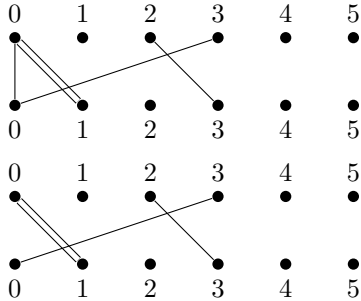
\begin{displaymath}
				\xymatrix@!=2.5pc{  
					\overset{0}{\bullet} \ar@{-}[d] \ar@2{-}[rd] & \overset{1}{\bullet}  & \overset{2}{\bullet}\ar@{-}[dr] &\overset{3}{\bullet}\ar@{-}[dlll] & \overset{4}{\bullet} &\overset{5}{\bullet}\\
					\underset{0}{\bullet} &  \underset{1}{\bullet} & \underset{2}{\bullet} & \underset{3}{\bullet} & \underset{4}{\bullet} & \underset{5}{\bullet}
				}
	\end{displaymath}

\begin{displaymath}
	\xymatrix@!=2.5pc{  
		\overset{0}{\bullet}  \ar@2{-}[rd] & \overset{1}{\bullet}  & \overset{2}{\bullet}\ar@{-}[dr] &\overset{3}{\bullet}\ar@{-}[dlll] & \overset{4}{\bullet} &\overset{5}{\bullet}\\
		\underset{0}{\bullet} &  \underset{1}{\bullet} & \underset{2}{\bullet} & \underset{3}{\bullet} & \underset{4}{\bullet} & \underset{5}{\bullet}
	}
\end{displaymath}
\end{example}	
	In the subsequent examples, we shall omit vertices with no edges incident on them to improve presentation. 
	
	The set $\tilde{\mathcal{B}}_{k}$ is in bijection with the set of multiset partitions of the multiset $\{1^k,1'^k\}$. 
	Explicitly, for a class $[\Gamma]\in \tilde{\mathcal{B}}_k$, let $\{(a_1,b_1),\ldots ,(a_n,b_n)\}$ be the multiset of non-zero edges of $\Gamma$ then we have the following bijective correspondence:
	\begin{align}\label{al:bij}
		[\Gamma] \leftrightarrow \{\{1^{a_i},1'^{b_i}\}\mid i=1,2,\ldots,n\}.
	\end{align} 
	Thus graphs in $\tilde{\mathcal B}_k$ are a diagrammatic interpretation of multiset partitions of $\{1^k,1'^k\}$.
	For example, the multiset partition of $\{1^5,1'^5\}$ associated with the class of graphs in Example \ref{graph} is $\{\{1'\},\{1'\},\{1^2,1'^{3}\},\{1^3\}\}$.

	Let $\MP_k(\xi)$ denote the free module over $F[\xi]$ with basis $\tilde{\mathcal{B}}_{k}$. We show how to multiply two basis elements in $\tilde{\mathcal{B}}_{k}$ by specifying the structure constants with respect to the basis $\tilde{\mathcal{B}}_{k}$.  To do this we define the necessary setup in the next few paragraphs. It is clear from the definition of these structure constants (see Equation~\eqref{eq:poly_gene}) that they are polynomials in $\xi$. 
	
	Let $[\Gamma_1], [\Gamma_2]\in \tilde{\mathcal{B}}_k$ and  pick the respective representatives $\Gamma_1,\Gamma_2$ with $n$ edges, where $n\geqslant \max\{\rn(\Gamma_1),\rn(\Gamma_2)\}$.  To compute the product $[\Gamma_1]*[\Gamma_2]$, we begin by concatenating $\Gamma_1$ with $\Gamma_2$ by
	placing $\Gamma_2$ above $\Gamma_1$ and identify the vertices in the bottom row of $\Gamma_2$ with those in the top row of $\Gamma_1$. A path on this concatenated diagram is read from top to bottom. Explicitly,  a \defn{path}  on this diagram is a sequence $(a,b,c)$ such that $(a,b)$ is an edge of $\Gamma_2$ and $(b,c)$ is an edge of $\Gamma_1$.
	
{\bf Configuration of paths.}	
		A \defn{configuration of $n$ paths} with respect to an ordered pair $(\Gamma_1,\Gamma_2)\in \mathcal B_k\times \mathcal B_k$ is a multiset $P=\{p_1,\ldots, p_n\}$, $p_i=(a_i,b_i,c_i)$, such that the covering condition holds:
		\begin{equation*}
			\begin{array}{l}
				E_{\Gamma_2}=\{(a_i,b_i)\mid   1\leqslant i\leqslant n\}, \text{and}\\
				E_{\Gamma_1}=\{(b_i,c_i)\mid 1\leqslant i\leqslant  n\}.
			\end{array}
		\end{equation*}
	
	For a configuration of $n$ paths $P=\{(a_1,b_1,c_1),\dotsc,(a_n,b_n,c_n)\}$, the rank of $P$ is denoted $\rn(P)$ and defined to be the number of paths $(a,b,c) \neq (0,0,0)$ in $P$. Let $\Gamma_P$ denote the multigraph with edge multiset $E_{\Gamma_P}=\{(a_1,c_1),\dotsc,(a_n,c_n)\}$, obtained by omitting the mid-points of paths in $P$. 

{\bf Support.}	The \defn{support}  $\Supp^n(\Gamma_1,\Gamma_2)$ is the set of all configurations of $n$ paths with respect to $(\Gamma_1 ,\Gamma_2)$.
	 For $[\Gamma]\in \tilde{\mathcal{B}}_{k}$ and $n\geqslant \max\{\rn(\Gamma_1), \rn(\Gamma_2),\rn(\Gamma)\}$, let $\Supp_{\Gamma}^{n}(\Gamma_1,\Gamma_2)$ be the set of configuration of $n$ paths $P$ in $\Supp^{n}(\Gamma_1,\Gamma_2)$ such that $\Gamma_P=\Gamma$.

	\begin{example}
		\label{eg:supp}
		Let $k=2$. Consider the following graph:    
		
		\begin{displaymath}\xymatrix{
				{} \ar @{}[d]|{\Gamma_1\quad=} &
				\overset{0}{\bullet} \ar@{-}[dr] & \overset{1}{\bullet} \ar@{-}[d] \ar@{-}[dl]\\
			{} &	\underset{0}{\bullet} & \underset{1}{\bullet}}
				\end{displaymath} 

		Let $\Gamma_2=\Gamma_1$.
		For $n\geqslant 3$, we add $n-3$ zero-weighted edges to each graph and identify the top row of $\Gamma_1$ with the bottom row of $\Gamma_2$.
		\begin{displaymath}\xymatrix{
			\overset{0}{\bullet} \ar@3{-}[d]|{n-3}\ar@{-}[dr] & \overset{1}{\bullet} \ar@{-}[d] \ar@{-}[dl]\\
				\underset{0}{\bullet}\ar@3{.}[d] & \underset{1}{\bullet}\ar@2{.}[d]\\
			\overset{0}	{\bullet} \ar@3{-}[d]|{n-3}\ar@{-}[dr] & \overset{1}{\bullet} \ar@{-}[d] \ar@{-}[dl]\\
				\underset{0}{\bullet}\ar@{-}[ur] & \underset{1}{\bullet}\\
			}
		\end{displaymath}
		Thus  $\Supp^{n}(\Gamma_1,\Gamma_2)$ comprises four elements $\{P_1,P_2,P_3,P_4\}$ where
		\begin{align*}
			P_1&=\{(0,0,0)^{n-4},(0,1,0),(0,0,1),(1,0,0),(1,1,1)\},\\
			P_2&=\{(0,0,0)^{n-3},(0,1,1),(1,1,0),(1,0,1)\},\\
			P_3&=\{(0,0,0)^{n-3},(0,1,0),(1,0,1),(1,1,1)\},\\ 
			P_4&=\{(0,0,0)^{n-4},(0,0,1),(0,1,1),(1,1,0),(1,0,0)\}.
		\end{align*}
		
		Note that $\Gamma_{P_1}=\Gamma_{P_2}$ and $\rn(P_1)=4=\rn(P_4)$, $\rn(P_2)=3=\rn(P_3)$. The graphs $\Gamma_{P_2}$, $\Gamma_{P_3}$ and  $\Gamma_{P_4}$, corresponding to  configurations $P_2,P_3$ and $P_4$ are represented by:
		
		\begin{displaymath}
			\xymatrix{ 
				{} \ar @{}[d]|{}&
			\overset{0}{\bullet} \ar@3{-}[d]|{n-3}\ar@{-}[dr] & \overset{1}{\bullet}\ar@{-}[d] \ar@{-}[dl]\\
				{}&  \underset{0}{\bullet} & \underset{1}{\bullet} 
			}
			\quad
			\xymatrix{ 
				{} \ar @{}[d]|{}&
				\overset{0}{\bullet} \ar@3{-}[d]|{n-2} & \overset{1}{\bullet}\ar@2{-}[d]\\
				{}&  \underset{0}{\bullet} & \underset{1}{\bullet} 
			}
			\quad
			\xymatrix{ 
				{} \ar @{}[d]|{}&
			\overset{0}{\bullet} \ar@3{-}[d]|{n-4}\ar@2{-}[dr] & \overset{1}{\bullet} \ar@2{-}[dl]\\
				{}& \underset{0}{\bullet} & \underset{1}{\bullet} 
			}	
		\end{displaymath}

		Observe that $\Supp^{n}_{\Gamma_{P_3}}(\Gamma_1,\Gamma_2)=\{P_3\}$ for $n\geqslant 3$ and $\Supp^{n}_{\Gamma_{P_4}}(\Gamma_1,\Gamma_2)=\{P_4\}$ for $n \geqslant 4$, and empty otherwise. Observe that $\Supp^{n}_{\Gamma_{P_2}}(\Gamma_1,\Gamma_2)=\{P_1,P_2\}$ for $n \geqslant 4$ but $\Supp^{3}_{\Gamma_{P_2}}(\Gamma_1,\Gamma_2)=\{P_2\}$.
	\end{example}

{\bf Structure constants.} 	Let $D_\Gamma$ be the set of distinct edges of $\Gamma$. For $(s,t) \in D_\Gamma$, define the multiset $P_{st}=\{b\mid(s,b,t) \in P\}$, the mid-points of paths in $P$ that starts at $s$ and ends at $t$. Let $p_{st}$ be the cardinality of $P_{st}$. For $r=0,1,\dotsc,k$, denote the multiplicity of the  path $(s,r,t)$ occurring in $P$ by $p_{st}(r)$.  \label{para:D}

	The \defn{structure constant} of $[\Gamma]$ in the product $[\Gamma_1]*[\Gamma_2]$ is the following:
	
	\begin{align}\label{eq:poly_gene}
		&\Phi^{[\Gamma]}_{[\Gamma_{1}][\Gamma_2]}(\xi)=\sum_{P \in \Supp_{\Gamma}^{3k}(\Gamma_1,  \Gamma_2)}\mathcal{K}_P \cdot (\xi-\rn(\Gamma))_{[P^{\Gamma_1\circ\Gamma_2}]},
		\text{ where} 
		\\ &\mathcal{K}_P=\frac{1}{p_{00}(1)!\cdots p_{00}(k)!}         
		\prod_{(s,t)\in D_\Gamma\setminus\{(0,0)\}} \binom{p_{st}}{p_{st}(0),\ldots,p_{st}(k)}, \nonumber
	\end{align}	
	and $[P^{\Gamma_1\circ\Gamma_2}]=\sum_{i=1}^{k}p_{00}(i)$, the number of paths $(0,i,0)$ in $P$ for all $i\geqslant 1$.

	\begin{remark}
		\label{rm:rank P}
		For $P \in \Supp^{n}(\Gamma_1,\Gamma_2)$, we observe that $\rn(P)= \rn(\Gamma_P) + [P^{\Gamma_1 \circ \Gamma_2}]$ and $\rn(P)\geqslant \max\{\rn(\Gamma_1),\rn(\Gamma_2)\}$.
	\end{remark}

	\begin{example}	
		\label{eg:struc_const}
		We compute $[\Gamma_1]* [\Gamma_2]$, where $\Gamma_1, \Gamma_2$ are as in Example \ref{eg:supp}. Note $\Gamma_{P_1}=\Gamma_{P_2}=\Gamma_1$. 
		
		By Equation \eqref{eq:poly_gene}, we have
		$$\mathcal{K}_{P_{1}}\cdot(\xi-\rn(\Gamma_1))_{[{P_1}^{\Gamma_1 \circ \Gamma_2}]}=\xi-3,$$
		as $[{P_1}^{\Gamma_1 \circ \Gamma_2}]=1$ and $\mathcal{K}_{P_{1}}=1$. 
		Again,  $$\mathcal{K}_{P_{2}}\cdot(\xi-\rn(\Gamma_1))_{[{P_2}^{\Gamma_1 \circ \Gamma_2}]}=1.$$
		Therefore, $\Phi^{[\Gamma_1]}_{[\Gamma_1] [\Gamma_2]}(\xi)=(\xi-3)+ 1$.

		A similar computation for $P_3, P_4$ yields $\Phi^{[\Gamma_{P_3}]}_{[\Gamma_1] [\Gamma_2]}(\xi)=2(\xi-2)$ and $\Phi^{[\Gamma_{P_4}]}_{[\Gamma_1] [\Gamma_2]}(\xi)=4$.
		Thus finally $[\Gamma_1]*[\Gamma_2]$ is equal to:
		
		\begin{displaymath}
			\xymatrix{ 
				{} \ar @{}[d]|{(\xi-2)}&
				\overset{0}{\bullet} \ar@{-}[dr] & \overset{1}{\bullet}\ar@{-}[d] \ar@{-}[dl]\\
				{}&  \underset{0}{\bullet} & \underset{1}{\bullet} 
			}
			\quad
			\xymatrix{ 
				{} \ar @{}[d]|{ + \quad   2(\xi-2)}&
				\overset{0}{\bullet}  & \overset{1}{\bullet}\ar@2{-}[d]\\
				{}&  \underset{0}{\bullet} & \underset{1}{\bullet} 
			}
			\quad
			\xymatrix{ 
				{} \ar @{}[d]|{+\quad     4}&
				\overset{0}{\bullet} \ar@2{-}[dr] & \overset{1}{\bullet} \ar@2{-}[dl]\\
				{}&  \underset{0}{\bullet} & \underset{1}{\bullet} 
			}	
		\end{displaymath}    
	\end{example}
	The product defined  in Equation~\eqref{eq:poly_gene} endows an algebra structure on $\MP_{k}(\xi)$ over $F[\xi]$. We call $\MP_k(\xi)$ the \defn{multiset partition algebra}.  For a partition $\lambda=(\lambda_1,\ldots,\lambda_l)$ of $k$, let $\Gamma_{\lambda}$ be the graph whose  multiset of edges is $E_{\Gamma_{\lambda}}=\{(\lambda_j,\lambda_j)\mid 1\leqslant j\leqslant l\}$. Consider the subset
	 $\mathcal{U}_k =\{[\Gamma_{\lambda}]\mid \lambda \text{ is a partition of } k\}$ of $\tilde{\mathcal{B}}_k$.
	  It is easy to verify that the element
	  \begin{equation}\label{mpa:id}
	 id:=\sum_{[\Gamma]\in\mathcal{U}_{k}}[\Gamma]
\end{equation}
is the \defn{identity} element of $\MP_k(\xi)$.
		\begin{example}
		\label{eg:identity}
		This example demonstrates multiplication by the identity element for $k=2$. By definition,
		
		\begin{displaymath}
			\xymatrix{ 
				{} \ar @{}[d]|{ id \quad =}&
				\overset{0}{\bullet} & \overset{1}{\bullet}\ar@2{-}[d]\\
				{}& \underset{0}{\bullet} & \underset{1}{\bullet} 
			}
			\quad
			\xymatrix{ 
				{} \ar @{}[d]|{ + }&
				\overset{0}{\bullet}  & \overset{2}{\bullet}\ar@1{-}[d]\\
				{}&  \underset{0}{\bullet} & \underset{2}{\bullet}}	
		\end{displaymath}
		Let $\Gamma_{(1,1)}$ and $\Gamma_{(2)}$ respectively denote the representatives with $3k$ edges, and consider the graph $\Gamma_1$ as below
		\begin{displaymath}
			\xymatrix{
				\overset{0}{\bullet} \ar@3{-}[d]|{3k-3}\ar@{-}[dr] & \overset{1}{\bullet} \ar@{-}[d] \ar@{-}[dl]\\
				\underset{0}{\bullet}\ar@{-}[ur] & \underset{1}{\bullet}\\
			}
		\end{displaymath}
		Then  $\Supp^{3k}(\Gamma_1,\Gamma_{(1,1)})= \{(0,0,0)^3,(0,0,1),(1,1,0),(1,1,1)\}$, while  $\Supp^{3k}(\Gamma_1,\Gamma_{(2)})$ is empty. The graph corresponding to this configuration is $\Gamma_1$.  
		Observe that $\Phi^{[\Gamma_1]}_{[\Gamma_1][\Gamma_\lambda]}(\xi)=1$, and $[\Gamma_1]*id=[\Gamma_1]$. Similarly, $ id*[\Gamma_1]=[\Gamma_1]$.
		\end{example}
	
	\begin{remark}
			For a non-negative integer $k$,  given $[\Gamma] \in \tilde{\mathcal{B}}_{k}$, $\Supp^{3k}(\Gamma,\Gamma_\lambda)$ is nonempty for the representative of a unique $[\Gamma_\lambda]$ in   $\mathcal{U}_{k}$. In this case, $\Supp^{3k}(\Gamma,\Gamma_\lambda)$  comprises a single configuration $P$, and $\Gamma_P=\Gamma$. This  helps us  conclude that $[\Gamma]*id=[\Gamma]=id*[\Gamma]$. 
		\end{remark}

 When $\xi$ is evaluated at a positive integer $n$ we prove  in Theorem~\ref{thm:swd_mult} that the algebra $\MP_k(n)$ is in Schur--Weyl duality with the symmetric group $S_n$ acting on $k$th symmetric power of $F^n$. As an application of Theorem~\ref{thm:swd_mult}, we show in Corollary~\ref{coro:asso} that $\MP_k(\xi)$ is associative. Proposition~\ref{thm:poly} is essential to prove this theorem.

	\begin{definition}
	\label{strc cons}
		For $n\geqslant\max\{\rn([\Gamma]),\rn([\Gamma_1]),\rn([\Gamma_2])\}$, let $\Gamma,\Gamma_1,\Gamma_2$ be the respective representatives with $n$ edges. 	Define 
		\begin{displaymath}
		C_{\Gamma_1 \Gamma_2}^{\Gamma}(n)=\big\{(b_1,\dotsc, b_n)\mid	\{(a_1,b_1,c_1),\ldots, (a_n,b_n,c_n)\} \in \Supp^n_{\Gamma}(\Gamma_1,\Gamma_2)\big\}
		\end{displaymath}
		where $(a_1,c_1)\leqslant \cdots   \leqslant(a_n,c_n)$ are the edges of $\Gamma$ labelled in lexicographically increasing order.
	\end{definition}
	
	\begin{proposition}\label{thm:poly}
		Given $[\Gamma],[\Gamma_1],[\Gamma_2] \in \tilde{\mathcal{B}}_k$ and $n \geqslant \max \{\rn(\Gamma),\rn(\Gamma_1),\rn(\Gamma_2)\}$, we have
		\begin{displaymath}
			\Phi^{[\Gamma]}_{[\Gamma_{1}][\Gamma_2]}(n)=|C_{\Gamma_1\Gamma_2}^{\Gamma}(n)|.
		\end{displaymath}
	\end{proposition}
	
The proof of the above proposition requires the following two lemmas.
	
	\begin{lemma}
		\label{th:cardC}
			Given $[\Gamma],[\Gamma_1],[\Gamma_2] \in \tilde{\mathcal{B}}_k$ and $n \geqslant \max \{\rn(\Gamma),\rn(\Gamma_1),\rn(\Gamma_2)\}$, we have
		\begin{equation}\label{eq:diset}
			|C_{\Gamma_1\Gamma_2}^{\Gamma}(n)|=\sum_{P \in \Supp_{\Gamma}^n({\Gamma_1\Gamma_2})}\mathcal{K}_P\cdot (n-\rn(\Gamma))_{[P^{\Gamma_1\circ\Gamma_2}]}.
		\end{equation}
	\end{lemma}

	\begin{proof}
		If $\Supp_{\Gamma}^n(\Gamma_1,\Gamma_2)$
		is empty, then $C^{\Gamma}_{\Gamma_1\Gamma_2}(n)=\emptyset$. 
		Otherwise, for $P \in \Supp_{\Gamma}^n(\Gamma_1,\Gamma_2)$, let $ D_\Gamma$ and $P_{st}$ be as defined just before Equation~\eqref{eq:poly_gene}. Let $(s_1,t_1),\dotsc,(s_j,t_j)$ be the elements of $D_\Gamma$ numbered in lexicographically increasing order.  Let $Perm(P_{st})$ be the set of distinct arrangements of the elements of $P_{st}$. Define 
		\begin{displaymath}
			\text{Arr}(P)= Perm(P_{s_1t_1})\times \dotsb \times Perm(P_{s_jt_j}).
		\end{displaymath}
		For each configuration $P$, $\text{Arr}(P)$ is a subset of $C^{\Gamma}_{\Gamma_1\Gamma_2}(n)$, and for distinct configurations $P,Q \in \Supp_{\Gamma}^n(\Gamma_1,\Gamma_2)$ $$\text{Arr}(P)\cap \text{Arr}(Q)=\emptyset.$$ Thus
		\begin{displaymath}
			C_{\Gamma_1\Gamma_2}^{\Gamma}(n)=\bigsqcup_{P \in \Supp_{\Gamma}^n({\Gamma_1\Gamma_2})}\text{Arr}(P).
		\end{displaymath}  
		Note that, $|Perm(P_{st})|=\binom{p_{st}}{p_{st}(0),\ldots, p_{st}(k)} = \frac{p_{st}!}{p_{st}(0)!p_{st}(1)!\cdots p_{st}(k)!}$. Hence
		\begin{equation*}\label{eq:poly_not_simple}
			|\text{Arr}(P)|=\prod_{(s,t)\in D_\Gamma} \binom{p_{st}}{p_{st}(0),\ldots, p_{st}(k)}. 
		\end{equation*}
		Also,  $p_{00}=n-\rn(\Gamma)= p_{00}(0) + [P^{\Gamma_1 \circ {\Gamma_2}}]$, thus 
		\begin{align*}
			\binom{p_{00}}{p_{00}(0),\ldots,p_{00}(k)}&=\binom{p_{00}}{p_{00}(0)}\binom{p_{00}-p_{00}(0)}{p_{00}(1),\dotsc,p_{00}(k)}\\
			&=\frac{(n-\rn(\Gamma))_{[P^{\Gamma_1\circ\Gamma_2}]}}{[P^{\Gamma_1\circ\Gamma_2}]!}\binom{p_{00}-p_{00}(0)}{p_{00}(1),\dotsc,p_{00}(k)}\\
			&=\frac{(n-\rn(\Gamma))_{[P^{\Gamma_1\circ\Gamma_2}]}}{p_{00}(1)!\cdots p_{00}(k)!}.
		\end{align*}
		So,
		\begin{displaymath}
			|\text{Arr}(P)|=\mathcal{K}_P \cdot (n-\rn(\Gamma))_{[P^{\Gamma_1\circ\Gamma_2}]},
		\end{displaymath}
		where $\mathcal{K}_{P}$ is as defined in Equation \eqref{eq:poly_gene}.
		 Thus $$|C_{\Gamma_1\Gamma_2}^{\Gamma}(n)|=\sum_{P \in \Supp_{\Gamma}^n({\Gamma_1\Gamma_2})}\mathcal{K}_P\cdot (n-\rn(\Gamma))_{[P^{\Gamma_1\circ\Gamma_2}]}.$$
	\end{proof}

One observes that $|C_{\Gamma_1\Gamma_2}^{\Gamma}(n)|$ in Equation \eqref{eq:diset} differs from the expression for $\Phi^{[\Gamma]}_{[\Gamma_{1}][\Gamma_2]}(n)$ in Equation \eqref{eq:poly_gene} only in the set over which the summation is carried out. The subsequent lemma establishes that this distinction is superfluous. 
\begin{lemma}
		\label{lem:supp}
		For $m \leqslant n$, the map
		\begin{align*}
			\beta_{m,n}:\Supp^{m}_{\Gamma}(\Gamma_1,\Gamma_2) &\to \Supp^{n}_{\Gamma}(\Gamma_1,\Gamma_2)\\
			P&\mapsto P\sqcup \{(0,0,0)^{n-m}\}\nonumber
		\end{align*}
		is an injection, where $\beta_{n,n}$ is the identity map. Moreover, for $m,n \geqslant 3k$, $\beta_{m,n}$ is surjective. 
	\end{lemma}
	\begin{proof}
		Adding extra $(0,0,0)$ paths is clearly an injective operation. For $m,n\geqslant 3k$, we prove the surjectivity of the map $\beta_{m,n}$. Note that it suffices to prove that $\beta_{3k,n}$ is surjective for $n \geqslant 3k$.  Given $P \in \Supp^{n}_{\Gamma}(\Gamma_1,\Gamma_2)$, $\rn(P) = \rn (\Gamma) + [P^{\Gamma_1 \circ \Gamma_2}]$  by Remark \ref{rm:rank P}. Since  the  weights of $\Gamma_1$ and $\Gamma_2$ are both $(k,k)$, $[P^{\Gamma_1 \circ \Gamma_2}] \leqslant k$ and $\rn(\Gamma) \leqslant 2k$.  
		Since $\rn(P) \leqslant 3k$,  $n-3k$ of the $(0,0,0)$ paths may be deleted from $P$ to give an element $\overline{P} \in \Supp^{3k}_{\Gamma}(\Gamma_1,\Gamma_2)$ such that $\beta_{3k,n}(\overline{P})=P$.
		\end{proof} 
	
	\begin{example}
		\label{ex:counter}
		For $n<3k$ the map $\beta_{n,3k}$ is not necessarily surjective. Consider $\Gamma=\Gamma_1=\Gamma_2$ as below:
		\begin{displaymath}
			\xymatrix{ 
				{} \ar @{}[d]|{}& 
			\overset{0}{\bullet} \ar@3{-}[d]|{n-2k}\ar@3{-}[drr]|(0.3){k} & & \overset{1}{\bullet} \ar@3{-}[dll]|(0.3){k}\\
				{}&  \underset{0}{\bullet} & & \underset{1}{\bullet} 
			}	
		\end{displaymath}
		Then for $n<3k$, $\Supp^{n}_{\Gamma}(\Gamma_1,\Gamma_2)= \emptyset$ but $$\Supp^{3k}_{\Gamma}(\Gamma_1,\Gamma_2)=\{(1,0,0)^k,(0,1,0)^k,(0,0,1)^k\}.$$
	\end{example}
	
\begin{proof}[\textbf{Proof of Proposition~\ref{thm:poly}}]\label{proof}
		Recall from Equation \eqref{eq:poly_gene}
		\begin{displaymath}
			\Phi^{[\Gamma]}_{[\Gamma_1][\Gamma_2]}(n)=\sum_{P \in \Supp_{\Gamma}^{3k}(\Gamma_1,  \Gamma_2)}\mathcal{K}_P \cdot (n-\rn(\Gamma))_{[P^{\Gamma_1\circ\Gamma_2}]}.
		\end{displaymath}
		This equals to
		\begin{displaymath}
			\sum_{\substack{P \in \Supp_{\Gamma}^{3k}(\Gamma_1,  \Gamma_2)\\ \rn(P)\leqslant n}}\mathcal{K}_P \cdot (n-\rn(\Gamma))_{[P^{\Gamma_1\circ\Gamma_2}]}
			\, +\sum_{\substack{P \in \Supp_{\Gamma}^{3k}(\Gamma_1,  \Gamma_2)\\ \rn(P)>n}}\mathcal{K}_P \cdot (n-\rn(\Gamma))_{[P^{\Gamma_1\circ\Gamma_2}]}.
		\end{displaymath}

		Using the arguments as in the proof of Lemma \ref{lem:supp}, observe that the sets $\Supp_{\Gamma}^n(\Gamma_1,\Gamma_2)$ and  \linebreak $\{P\in \Supp_{\Gamma}^{3k}(\Gamma_1,\Gamma_2)\mid \rn(P)\leqslant n\}$ are in bijection. Since $\mathcal{K}_P$, $[P^{\Gamma_1\circ\Gamma_2}]$ are unaffected by the addition or removal of $(0,0,0)$ paths,   $|C_{\Gamma_1\Gamma_2}^{\Gamma}(n)|$ is equal to the first sum by Lemma~\ref{th:cardC}. It remains to prove the second sum  is zero.
		
		For $n>3k$, $\{P\in \Supp_{\Gamma}^{3k}(\Gamma_1,\Gamma_2)\mid \rn(P)> n\}$ is empty by Lemma \ref{lem:supp}. For $n< 3k$, consider $P \in \Supp_{\Gamma}^{3k}(\Gamma_1,  \Gamma_2)$ with $\rn(P)>n$. By Remark \ref{rm:rank P}, $\rn(P)=\rn(\Gamma)+[P^{\Gamma_1\circ\Gamma_2}]$. Thus
		$$n<\rn(\Gamma)+[P^{\Gamma_1\circ\Gamma_2}],$$
		so $ (n-\rn(\Gamma))_{[P^{\Gamma_1\circ\Gamma_2}]}=0$ for every path $P$ that lies in $\Supp_{\Gamma}^{3k}(\Gamma_1,  \Gamma_2)$ with $\rn(P)>n$.
	\end{proof}
	
	\subsection{Schur--Weyl duality}\label{sec:tens_sym} Let $\ZZ_{\geqslant 0}$ denote the set of all non-negative integers. 
	For a fixed positive integer $n$ and a fixed $k\in\ZZ_{\geqslant 0}$, the defining action of $GL_n(F)$ on $F^n$ extends uniquely to 
	the $k$th symmetric power $\Sym^{k}(F^n)$.
	So  the subgroup $S_n$ of $GL_n(F)$ acts on $\Sym^{k}(F^n)$ by the restriction.  Let 
	\begin{displaymath}
		M(n,k)=\{\mathsf{a}:=(a_1,\ldots,a_n)\mid  \sum_{i=1}^{n} a_i=k,\,a_i\in\ZZ_{\geqslant 0}\},
	\end{displaymath}
	the set of weak 
	compositions of $k$ into $n$ parts.
	The set $$\{e^{\mathsf{a}}:=e_1^{a_1}\cdots e_n^{a_n}\in \Sym^k(F^n)\mid \mathsf{a}=(a_1,\ldots,a_n)\in M(n,k)\}$$
	is a basis of $\Sym^k(F^n)$. For $w\in S_n$ and $\mathsf{a}\in M(n,k)$ define an action of $S_n$  on $M(n,k)$ by
	$$w.\mathsf{a}=(a_{w^{-1}(1)},\ldots,a_{w^{-1}(n)}).$$ 
	
	The map $e^{\mathsf{a}} \mapsto 1_{\mathsf{a}}$, where $1_{\mathsf{a}} $ is the indicator function of $\mathsf{a}\in M(n,k)$, determines the isomorphism of $S_n$-modules:
	\begin{equation}\label{eq:sym_k}
		\Sym^{k}(F^n) \cong F[M(n,k)],
	\end{equation}
where $F[M(n,k)]$ is the permutation representation of $S_n$ (see Section \ref{sec:twist-perm-repr}).	Thus we can regard $\Sym^k(F^n)$ as a permutation representation of $S_n$. We have the following algebra isomorphism 
	\begin{equation}\label{eq:cent_alg}
		\End_{S_n}(\Sym^k(F^n)) \cong \End_{S_n}(F[M(n,k)]).
	\end{equation}
	We express an element $(\mathsf{a}, \mathsf{b})\in M(n,k)\times M(n,k)$ as the biword: 
	\begin{align}\label{al:b}
	\begin{bmatrix}
		a_1       & a_2 & \dots & a_{n} \\
		b_1       & b_2 & \dots & b_n
	\end{bmatrix}. 
\end{align}
	The action of $S_n$ on $(\mathsf{a}, \mathsf{b})$ corresponds to permuting the columns of its biword.  Let $S_n\backslash (M(n,k)\times M(n,k))$ 
	be the set of  orbits of this action. Each orbit is uniquely determined by a biword of the form \eqref{al:b}, with its columns arranged in lexicographically increasing order. Define $\tilde{\mathcal B}_{k,n}:=\{[\Gamma]\in \tilde{\mathcal{B}}_k \mid \rn ([\Gamma]) \leqslant n\}$.
	Consider the map \begin{align} \label{k-orbit-diagram bij}
		S_n\backslash (M(n,k)\times M(n,k))& \rightarrow \tilde{\mathcal B}_{k,n} \\\text{defined by}
		\begin{bmatrix}
			a_1       & a_2 & \dots & a_{n} \\
			b_1       & b_2 & \dots & b_n
		\end{bmatrix}&\mapsto [\Gamma_{\mathsf{a},\mathsf{b}}], \nonumber
	\end{align}
	where $\Gamma_{\mathsf{a},\mathsf{b}}$ is the graph with multiset of edges $\{ (a_i, {b_i})\mid i=1,2,\dots,n \}.$ 
	
	The map is well-defined because $\Gamma_{\mathsf{a},\mathsf{b}} $ is unaffected by a permutation of the columns of the biword. For $[\Gamma]\in \tilde{\mathcal B}_{k,n}$, one can choose the representative $\Gamma'$ with $n$ edges.
	The multiset of edges of $\Gamma'$ 
	uniquely determines an $S_n$-orbit of $M(n,k)\times M(n,k)$. Hence the map \eqref{k-orbit-diagram bij} is a bijection.

	Following Theorem \ref{corollary:twisted-intertwiner} we define the integral operator corresponding to an $S_n$-orbit of $M(n,k)\times M(n,k)$.
	
	\begin{definition} \label{def:integral operator}
		For each $[\Gamma] \in \tilde{\mathcal B}_{k,n}$,   define $T_{[\Gamma]} \in \End_{S_n}(F[M(n,k)])$ by 
		$$ T_{[\Gamma]} (1_{\mathsf{a}})=\sum_{\{\mathsf{b} \in M(n,k) \mid[\Gamma_{\mathsf{a},\mathsf{b}}]=[\Gamma]\}}1_{\mathsf{b}}, \quad\text{ where } \mathsf{a}, \mathsf{b}\in M(n,k).$$
	\end{definition}
	\begin{example}\label{ex:T}
		For $k=2$ and $n=3$, consider the following graph:
		
		
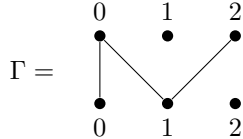
\begin{displaymath}
			\xymatrix{
			{}\ar@{}[d]|{\Gamma\quad=} &	\overset{0}{\bullet} \ar@{-}[d]\ar@{-}[dr] & \overset{1}{\bullet}  & \overset{2}{\bullet}  \ar@{-}[dl]\\
			{} &	\underset{0}{\bullet} & \underset{1}{\bullet} & \underset{2}{\bullet}\\
			}
		\end{displaymath}
		We compute the action of the integral operator $T_{[\Gamma]}$ on some basis elements of $F[M(3,2)]$.
		
		For $\mathsf{a}=(2,0,0)$,  we must choose tuples $\mathsf{b}$ such that the pairs $(a_i,b_i)$ exhaust all the edges of the graph $\Gamma$. There are two such tuples: $(1,1,0)$ and $(1,0,1)$ which can easily be verified to yield the multiset $\{(0,0),(0,1),(2,1)\}$ of edges. So $T_{[\Gamma]}(1_{(2,0,0)})=1_{(1,1,0)} + 1_{(1,0,1)}$. Similarly, $T_{[\Gamma]}(1_{(0,2,0)})=1_{(1,1,0)} + 1_{(0,1,1)}$ and  $T_{[\Gamma]}(1_{(0,0,2)})=1_{(1,0,1)} + 1_{(0,1,1)}$.
		
	Note that $T_{[\Gamma]}(1_{\mathsf{a}})=0$ for every other element in $\mathsf{a}\in M(3,2)$ because there are no admissible tuples $\mathsf{b}$ such that $\Gamma_{\mathsf{a},\mathsf{b}}=\Gamma$.
	
	\end{example}
	
	\begin{proposition}
		\label{thm:strcu_const}
			The set $\{T_{[\Gamma]} \mid [\Gamma] \in \tilde{\mathcal B}_{k,n}  \}$ is a basis of $\End_{S_n}(F[M(n,k)])$.
		Given  $[\Gamma_1],[\Gamma_2] \in \tilde{\mathcal{B}}_{k,n}$,  the product is given by
		$$ T_{[\Gamma_1]} T_{[\Gamma_2]}=\sum_{[\Gamma] \in \tilde{ \mathcal{B}}_{k,n}} |C_{\Gamma _1 \Gamma _2}^{\Gamma}(n)|T_{[\Gamma]},
		$$
	 where $\Gamma,\Gamma_1,\Gamma_2$ are the representative of the respective classes with $n$ edges and $C_{\Gamma _1 \Gamma _2}^{\Gamma}(n)$ is defined in Definition~\ref{strc cons}.
	\end{proposition}
	\begin{proof} 
	The first assertion follows from the isomorphism~\eqref{eq:cent_alg}, the bijection \eqref{k-orbit-diagram bij}, and
	Theorem \ref{corollary:twisted-intertwiner}. 
		
		Let $(\mathsf{a}$, ${\mathsf{c}})\in M(n,k)\times M(n,k)$ such that $[\Gamma_{\mathsf{a},\mathsf{c}}]=[\Gamma]$. Then
		\begin{align*}
			T_{[\Gamma_1]} T_{[\Gamma_2]}(1_{\mathsf{a}})&=T_{[\Gamma_1]}\bigg(\sum_{\{\mathsf{b}\mid[\Gamma_{\mathsf{a},\mathsf{b}}]=[\Gamma_2]\}}1_{\mathsf{b}}\bigg)\\ & =\sum_{\{\mathsf{b}\mid[\Gamma_{\mathsf{a},\mathsf{b}}] =[\Gamma_2]\}}T_{[\Gamma_1]}(1_{\mathsf{b}})\\ &=\sum_{\{\mathsf{b}\mid[\Gamma_{\mathsf{a},\mathsf{b}}]=[\Gamma_2]\}}\sum_{\{\mathsf{z}\mid[\Gamma_{\mathsf{b},\mathsf{z}}]=[\Gamma_1]\}}1_{\mathsf{z}}.
		\end{align*}
	So the coefficient of $T_{[\Gamma]}$ in  $T_{[\Gamma_1]}T_{[\Gamma_2]}$ is the
	 cardinality of $\{{\mathsf{b}}\in M(n,k)\mid[\Gamma_{\mathsf{a},\mathsf{b}}]=[\Gamma_2] , [\Gamma_{\mathsf{b},\mathsf{c}}] =[\Gamma_1] \}$. For each $\mathsf{b}$ in this set, $\{(a_1,b_1,c_1),\dotsc,(a_n,b_n,c_n)\} \in \Supp^n_{\Gamma}(\Gamma_1,\Gamma_2)$, so $\mathsf{b} \in C^{\Gamma}_{\Gamma_1\Gamma_2}(n)$. Conversely, each element $\mathsf{b} \in C^{\Gamma}_{\Gamma_1\Gamma_2}(n)$ is in the above set. 
	\end{proof}
\textbf{Schur--Weyl duality for the multiset partition algebra.}	In the seminal paper \cite{Jones}, Jones proved the Schur--Weyl duality (Theorem~\ref{thm:SWD_partition}) between the actions of 
	$\Pa_k(n)$ and $S_n$ on $(F^n)^{\otimes k}$. In this section, we give an analog of the above scenario for the actions of $\MP_{k}(n)$ and $S_n$  on $F[M(n,k)]$. As the space $F[M(n,k)]$ is isomorphic to the $k$-th symmetric power space $\Sym^k(F^n)$,  the following theorem can be viewed as Schur--Weyl duality between $\MP_k(n)$ and $S_n$ acting on  $\Sym^{k}(F^n)$.
	\begin{theorem}\label{thm:swd_mult}
		The map 
		\begin{equation}
			\label{eq:swd}
			\phi: \MP_k(n) \to   \End_{S_n}(F[M(n,k)])  
		\end{equation}
		\[\phi([\Gamma])=
		\begin{cases}
			T_{[\Gamma]} &\text{ if $\rn[\Gamma]\leqslant n$},\\
			0 &\text{otherwise}
		\end{cases}
		\]
	 is a surjective  algebra homomorphism with the kernel
		$$\ker (\phi)=F\text{-$\spn$}\{[\Gamma]\in\tilde{\mathcal{B}}_{k}\mid \rn([\Gamma])>n\}.$$ In particular, when $n\geqslant 2k$, 
		$$\MP_{k}(n) \cong \End_{S_n}(F[M(n,k)]).$$
		Moreover, $F[S_n]$ generates $\End_{\MP_k(n)}(F[M(n,k)])$.
	\end{theorem}

	\begin{proof}
		From Proposition~\ref{thm:strcu_const}, $\{T_{[\Gamma]}\mid [\Gamma]\in\tilde{\mathcal{B}}_{k,n}\}$ is a basis of $\End_{S_n}(F[M(n,k)])$, so by the definition of $\phi$ it is a surjective linear map.
		
		Let $[\Gamma_1]$ and $[\Gamma_2]\in \tilde{\mathcal{B}}_{k}$.
		If $[\Gamma_1],[\Gamma_2]\in\tilde{\mathcal{B}}_{k,n}$,
		\begin{align*}
			\phi([\Gamma_1]*[\Gamma_2])&= \phi\bigg(\sum_{[\Gamma]\in\tilde{\mathcal{B}}_k}\Phi^{[\Gamma]}_{[\Gamma_1][\Gamma_2]}(n)[\Gamma]\bigg) \text{ by Equation \eqref{eq:poly_gene}}\\
			&=\sum_{[\Gamma]\in\tilde{\mathcal{B}}_{k,n}}\Phi^{[\Gamma]}_{[\Gamma_1][\Gamma_2]}(n) T_{[\Gamma]} \text{ by Definition~\ref{def:integral operator}}\\
			&= T_{[\Gamma_1]} T_{[\Gamma_2]} \text{ from Proposition \ref{thm:strcu_const} and Proposition \ref{thm:poly}}.
		\end{align*}
		If $\rn(\Gamma_1)> n$ or 
		$\rn(\Gamma_2)> n$ and $[\Gamma]\in \tilde{\mathcal{B}}_{k}$ with $\rn(\Gamma)\leqslant n$ appears in $[\Gamma_1]*[\Gamma_2]$,
		then we claim that $(n-\rn(\Gamma))_{[P^{\Gamma_1\circ\Gamma_2}]}=0$ for all $P\in\Supp^{3k}_{\Gamma}(\Gamma_1,\Gamma_2)$. Assume without loss of generality that the $\rn(\Gamma_1)>n$. Then for any $P \in \Supp^{3k}_{\Gamma}(\Gamma_1,\Gamma_2)$,
		$n-\rn(\Gamma)=n-\rn(P) + [P^{\Gamma_1 \circ \Gamma_2}]$ and $\rn(P)\geqslant \rn(\Gamma_1)$ by Remark \ref{rm:rank P}. Therefore $n-\rn(\Gamma)< [P^{\Gamma_1 \circ \Gamma_2}]$. Thus $$\phi([\Gamma_1]*[\Gamma_2])=0=\phi([\Gamma_1])\phi([\Gamma_2]).$$
		Note that $[\Gamma]\in \tilde{\mathcal{B}}_{k}$ has rank at most $2k$, so when $n\geqslant 2k$, $\ker(\phi)$ is trivial and thus $\phi$ is an isomorphism.
	\end{proof}

	\begin{corollary}[Associativity]\label{coro:asso}
		For $[\Gamma_1],[\Gamma_2],$ and $[\Gamma_3]\in \tilde{\mathcal{B}}_{k}$, 
		\begin{displaymath}
			([\Gamma]_{1} *[\Gamma_2])*[\Gamma_3]=[\Gamma_1]*([\Gamma_2]*[\Gamma_3]).
		\end{displaymath}
	\end{corollary}
	\begin{proof}
		For $[\Gamma]\in \tilde{\mathcal{B}}_{k}$, the coefficients of $[\Gamma]$ both in $([\Gamma]_{1} *[\Gamma_2])*[\Gamma_3]$
		and $[\Gamma_1]*([\Gamma_2]*[\Gamma_3])$ are polynomials in $\xi$. From Theorem~\ref{thm:swd_mult},  the values of 
		these  polynomials are equal at any positive integer $n\geqslant 2k$ because of the associativity of the product in $\End_{S_n}(F[M(n,k)])$. Since the characteristic of $F$ is $0$, these polynomials must be the same.
	\end{proof}

 The above corollary completes the proof of the following theorem. 
 
	\begin{theorem}\label{thm:mult_k}
		For $[\Gamma_1],[\Gamma_2]$ in $\tilde{\mathcal{B}}_k$
		we define the operation:
		\begin{equation}\label{eq:ass}
			[\Gamma_{1}]*[\Gamma_2]=
			\sum_{[\Gamma]\in \tilde{\mathcal{B}}_k} \Phi_{[\Gamma_1][\Gamma_2]}^{[\Gamma]}(\xi)[\Gamma],
		\end{equation}
		where $\Phi_{[\Gamma_1][\Gamma_2]}^{[\Gamma]}(\xi)\in F[\xi]$ is  given in Equation~\eqref{eq:poly_gene}.
		The linear extension of this operation makes $\MP_k(\xi)$ an associative unital algebra over $F[\xi]$, with the identity element $id=\sum_{[\Gamma]\in\mathcal{U}_k}[\Gamma]$.
	\end{theorem}

	\begin{remark}\label{coro:mult_irrep}
		For $n\geqslant 2k$, it follows from the centralizer theorem (\cite[Theorem~5.4]{HR}) that the irreducible representations of $\MP_k(n)$ correspond to the irreducible representations of $S_n$ appearing in $\Sym^k(F^n)$. Moreover, the multiplicity of an irreducible representation of $S_n$ occurring in $\Sym^k(F^n)$ is equal to the dimension of the corresponding irreducible representation of $\MP_k(n)$. In the next section we describe the indexing set for irreducible representations of $\MP_k(n)$. 
	\end{remark}	
	
	\section{Representation theory of the multiset partition algebra}
Given a partition $\nu=(\nu_1\geqslant \nu_2\geqslant \cdots)$ of  $n$ (denoted by $\nu \vdash n$), the \defn{Specht module} $V_\nu $ has a combinatorial basis, known as Young's seminormal basis \cite{Young}, consisting of standard Young tableaux of shape $\nu$. In this spirit, we may observe the dimension of an irreducible representation of $\MP_k(n)$ is obtained by counting certain semistandard multiset tableaux. Let $\Lambda_{k,n}$ denote the subset of partitions of $n$ that index the irreducible representations of $S_n$ occurring in $\Sym^k(F^n)$. We know from Remark \ref{coro:mult_irrep} that this set also indexes the irreducible representations of $\MP_k(n)$. In Theorem \ref{coro:mult} we describe the set $\Lambda_{k,n}$. 

	For every vector $\lambda=(\lambda_1,\lambda_2,\ldots,\lambda_s)$ of non-negative integers, consider the following representation of $S_n$: 
	\begin{equation}\label{eq:sym_lambda}
		\Sym^{\lambda}(F^n):=\Sym^{\lambda_1}(F^n)\otimes \Sym^{\lambda_2}(F^n)\otimes\cdots\otimes \Sym^{\lambda_s}(F^n).
	\end{equation}
	Let $a_{\nu}^{\lambda}$ denote the multiplicity of $V_{\nu}$ in $\Sym^{\lambda}(F^n)$. In the following proposition, we paraphrase Proposition 3.11 of \cite{Nate} to express $a_{\nu}^{\lambda}$ in the context of multiset tableaux.
	\begin{proposition}\label{nate}
		The multiplicity $a_{\nu}^{\lambda}$ is the number of semistandard multiset tableaux of shape $\nu$ and content $\{1^{\lambda_1},2^{\lambda_2},\ldots,s^{\lambda_s}\}$.
	\end{proposition}
	
	When $\lambda=(k)$ and $n\geqslant 2k$, we see from the previous proposition and Remark \ref{coro:mult_irrep} that the dimension of the irreducible representation corresponding to a partition $\nu$ is equal to the number of semistandard multiset tableaux of shape $\nu$ and content $\{1^k\}$.

Recall that the \defn{Schur function} corresponding to $\nu$ is defined (see \cite{MR1676282}) as  
	$$s_{\nu}(x_1,x_2,\ldots )= \sum_{T }x^T,$$ 
	where the sum is over all the semistandard Young tableau (SSYT) $T$ of shape $\nu$, and   $x^T:=x_1^{c_1(T)}x_2^{c_2(T)}\dotsb$. Here        $c_i(T)$ is defined to be the number of occurrences of a positive integer $i$ in $T$. The Schur functions form a $\mathbb{Z}$-basis for the ring of symmetric functions.                                                                                   
	
	For symmetric functions $f$ and $g$, Littlewood~\cite{littlewood_1958} defined the \defn{plethysm} $f[g]$ (for a more rigorous definition see~\cite[Section I.8]{MR3443860}) as a substitution of the multiset of monomials of $g$ into the variables of $f$. The next lemma gives a generating function for $a_{\nu}^{\lambda}$ as a plethysm of symmetric functions.
	\begin{lemma}
		\label{lm:gen_lambda}
		Let $h_i=h_i(q_1,q_2,\dotsc)$ be the complete homogeneous symmetric function of degree $i$ in the variables $q_1,q_2,\ldots$. For $\lambda=(\lambda_1,\lambda_2,\ldots,\lambda_s)$, let $q^{\lambda}=  q_1^{\lambda_1}q_2^{\lambda_2}\cdots q_s^{\lambda_s}$. Then
		\begin{displaymath}
			s_{\nu}[1+h_1+h_2+\cdots]=\sum_{\lambda}a_{\nu}^{\lambda}q^{\lambda},
		\end{displaymath}
		where $\lambda$ runs over all vectors of non-negative integers.
	\end{lemma}
	\begin{proof}
		Each monomial in the Schur function $s_{\nu}(x_1,x_2,\ldots)$ corresponds to an SSYT $T$, and the degree of $x_i$ in this monomial is $c_i(T)$. 
		
		By definition of plethystic substitution, $s_{\nu}[1+h_1+h_2+\cdots]$ can be regarded as the bijective substitution of the monomials of $1+h_1+h_2+\cdots$ into the variables $x_1,x_2,\ldots$ of $s_{\nu}$.  The monomials of the expression $1+h_1+h_2+\cdots$ are indexed by multisets of finite size (in particular the monomials of $h_k$ are indexed by multisets of size $k$), and the variables of $s_{\nu}$ are indexed by positive integers. We have natural order on the positive integers and graded lexicographic order on the multisets of finite size (see Definition \ref{def:go}). There exists a unique order preserving bijection, denoted $\phi$, from the positive integers to the multisets of finite size.   
		
		Given a positive integer $i$, let $\phi(i)=\{1^{\phi(i)_{1}},2^{\phi(i)_2},\ldots\}$. Under this map $x_i \mapsto q_1^{\phi(i)_1}q_2^{\phi(i)_2}\dotsb$. Replacing the integers in a semistandard Young tableaux with their images under $\phi$ gives a semistandard multiset tableaux. We have
		$$s_{\nu}[1+h_1+h_2+\cdots]=\sum_{T} (q_1^{\phi(1)_1}q_2^{\phi(1)_2}\cdots)^{c_1(T)}(q_1^{\phi(2)_1}q_2^{\phi(2)_2}\cdots)^{c_2(T)}\cdots,$$
		where the sum is over all the semistandard Young tableau (SSYT) $T$ of shape $\nu$.    Thus $s_{\nu}[1+h_1+h_2+\cdots]$ is the generating function for the multiset semistandard tableaux of shape $\nu$ by their content. The coefficient of $q^{\lambda}$ in this function is thus the number of semistandard multiset tableaux of shape $\nu$ and content $\{1^{\lambda_1},\ldots,s^{\lambda_s}\}$ which is precisely $a_{\nu}^{\lambda}$.   
	\end{proof}
	 
	\begin{theorem}\label{coro:mult}
		Let $\Lambda_{k,n}$ be the set $\{\nu \vdash n \mid a_\nu^{(k)} \neq 0\}$. Then
	 $$\Lambda_{k,n}=\{\nu \vdash n \mid \sum_{i=1}^n (i-1)\nu_i \leqslant k \}.$$
		\end{theorem}
	\begin{proof}
		 By taking $q_1=q$ and $q_i=0$ for all $i\geqslant 2$ in Lemma \ref{lm:gen_lambda}, we get
		\begin{displaymath} 
			s_{\nu}(1,q,q^2,\ldots)=\sum_{k\geqslant 0}a_{\nu}^{(k)}q^{k}.
		\end{displaymath} 
		From \cite[Corollary 7.21.3]{MR1676282}, we have the identity: 
		\begin{equation}
			s_{\nu}(1,q,q^2,\ldots) = \frac {q^{b(\nu)}}{\prod_{u \in \nu}(1-q^{h(u)})},
		\end{equation}
		where $b(\nu) = \sum_{i=1}^n (i-1)\nu_i$ and, for a cell $u$ in the Young diagram of $\nu$, $h(u)$ is the hook-length of $u$.                                                                                                                                         
		
		Thus the multiplicity of the Specht module $V_{\nu}$ in $\Sym^{k}(F^n)$ is the coefficient of $q^k$ in $\frac{q^{b(\nu)}}{\prod_{u \in \nu}(1-q^{h(u)})}$. Each term $(1-q^{h(u)})$ in the denominator may be expanded as $1+ q^{h(u)}+(q^{h(u)})^2+\cdots$. Thus the lowest exponent of $q$ in $\frac{q^{b(\nu)}}{\prod_{u \in \nu}(1-q^{h(u)})}$ is $b(\nu)$. If $b(\nu)>k$ then the coefficient of $q^k$ is zero, and thus $\nu \notin \Lambda_{k,n}$. 
		
		Conversely if $b(\nu)\leqslant k$, pick a corner cell $u_0$ (hook-length $1$). Then $\frac{1}{1-q^{h(u_0)}}=1+q+q^2+\cdots$, and 
		\begin{align*}
			\frac{q^{b(\nu)}}{\prod_{u \in \nu}(1-q^{h(u)})}= q^{b(\nu)}(1+q+q^2+\cdots)\prod_{u \neq u_0}(1+ q^{h(u)}+(q^{h(u)})^2+\cdots).
		\end{align*}
		Choosing the term $1$ from each of the infinite sums corresponding to cells $u \neq u_0$ and the term $q^{k-b(\nu)}$ from the infinite sum corresponding to the cell $u_0$ yields an instance of $q^k$. Since the expression above is a positive sum, demonstrating a single instance is sufficient to prove the positivity of the coefficient of $q^k$. Thus $a_\nu^{(k)} \neq 0$ if and only if $b(\nu)\leqslant k$.
		\end{proof}

	\subsection*{Restriction from general linear group to symmetric group}
	The irreducible polynomial representations $W_{\lambda}$ of $GL_n(F)$ of degree $d$ are indexed by the partitions $\lambda$ of $d$ with at most $n$ parts. 
	Consider the decomposition of the restriction of $W_{\lambda}$ to $S_n$ into Specht modules:
	\begin{displaymath}
		\Res^{GL_n(F)}_{S_n} W_\lambda = \bigoplus_{\nu\vdash n} V_\nu^{\oplus r^\lambda_\nu}.
	\end{displaymath}
	A combinatorial interpretation of these multiplicities remains an open problem. By Littlewood's formula~\cite{littlewood_1958}:
	\begin{displaymath}
		r^{\lambda}_\nu=\langle s_{\nu}[1+h_1+h_2+\cdots],s_{\lambda}\rangle,
	\end{displaymath}
	where $\langle-,-\rangle$ denotes the usual inner product on the ring of symmetric functions. 
	 Lemma \ref{lm:gen_lambda} allows us to express Littlewood's identity as a generating function. This mirrors a method described in \cite[Section 4.3]{Nate} to calculate these coefficients.
	\begin{proposition}
		\label{theorem:r-gen}
		Let $h_i=h_i(q_1,\dotsc,q_n)$ be the complete symmetric functions in $n$ variables. For every partition $\lambda=(\lambda_1,\lambda_2, \ldots,\lambda_n)$ of $d$ with at most $n$ parts and every partition $\nu$ of $n$, $r^\lambda_\nu$ is the coefficient of $q^\lambda:=q_1^{\lambda_1}\ldots q_n^{\lambda_n}$ in
		\begin{displaymath}
			s_\nu[1+h_1+\cdots]\prod_{i<j}(1-\frac{q_j}{q_i}).
		\end{displaymath}
	\end{proposition}
	
	\begin{proof}
		By Lemma~\ref{lm:gen_lambda}, $\sum_\mu a^\mu_\nu q^\mu = s_\nu[1+h_1+\dotsb] $,  where $\mu$ is a  non-negative integer vector with at most $n$ parts.
		
		The Jacobi--Trudi identity can be expressed at the level of the Grothendieck ring of the category of polynomial representations of $GL_n(F)$ (see \cite{MR1194310}) as:
		\begin{equation}
			\label{eq:jt}
			W_\lambda = \det(\Sym^{\lambda_i+j-i}(F^n)),
		\end{equation}
		so
		\begin{align*}
			r^\lambda_\nu & =\sum_{w\in S_n} \sgn(w)a^{w\cdot \lambda}_\nu\\
			& = \sum_{w\in S_n} \sgn(w) [q^{w\cdot \lambda}]s_\nu[1+h_1+\cdots]\\
			& = [q^\lambda] \sum_{w\in S_n} \sgn(w) s_\nu[1+h_1+\cdots]\prod_{i=1}^n q_i^{i-w(i)}\\
			& = [q^\lambda] s_\nu[1+h_1+\cdots] \sum_{w\in S_n} \sgn(w)\prod_{i=1}^n q_i^{i-w(i)}\\
			& =  [q^\lambda] s_\nu[1+h_1+\cdots]\det(q^{i-j}_i)\\
			&=[q^\lambda] s_\nu[1+h_1+\cdots](\prod_{i=1}^{n}q_{i}^{i-1})\det(q_{i}^{1-j}),
		\end{align*}
		where $w\cdot\lambda$ is the tuple whose $i$th coordinate is $\lambda_i+w(i)-i$, and $[q^{\lambda}]f(q_1,\ldots,q_n)$ is the coefficient of $q^{\lambda}$ in the function $f(q_1,\ldots,q_n)$.
		Recognizing  $\det(q_{i}^{1-j})$ as a Vandermonde determinant that evaluates to $\prod_{i<j}(q_j^{-1}-q_i^{-1})$ and simplifying gives Proposition~\ref{theorem:r-gen}.
	\end{proof}
	
	\section{Structural properties of the multiset partition algebra}\label{sec:five}
 The goal of this section is to give an embedding of $\MP_k(\xi)$ in $\Pa_k(\xi)$. Using this embedding and the cellularity of $\Pa_k(\xi)$, we prove that $\MP_k(\xi)$ is cellular. 

	\subsection{An embedding of $\MP_{k}(\xi)$ into $\Pa_{k}(\xi)$}
\label{subsec:embed}  

Let $d=\{B_1,\dotsc,B_s\}$ be a partition diagram in $\mathcal{A}_k$. Recall from Definition \ref{eg:par_alg_diag} that $B_i^u$ and $B_i^l$ are the primed and unprimed numbers in $B_i$ respectively, for all $i=1,\dotsc,s$.
 Let $\mathsf{G}_d$ be the multigraph in $\mathcal{B}_{k}$ with multiset of edges
 \begin{displaymath}
 E_\Gamma=\{(|B_i^u|,|B_i^l|)\mid 1\leqslant i \leqslant s\}.
 \end{displaymath}
For example, the diagram $d=\{\{1,2,1',3'\},\{3,5,4'\},\{4,2',5'\}\}$ in Example \ref{eg:block} corresponds to the graph $\mathsf{G}_d$ with multiset of edges $\{(2,2),(2,1),(1,2)\}$. 

Given $[\Gamma] \in \tilde{\mathcal{B}}_{k}$, define the set $O_\Gamma$ of {partition diagrams associated with $[\Gamma]$} as
\begin{displaymath}
O_\Gamma= \{d \in \mathcal{A}_k \mid  [\mathsf{G}_d]= [\Gamma]\}.
\end{displaymath}
We now define the canonical partition diagram associated to $[\Gamma]$. Pick a representative of $[\Gamma]$ with non-zero edges $ \{(a_1,b_1),\dotsc, (a_s,b_s)\}$ (the edges have been numbered in lexicographically increasing order), and define the following subset of $\{1,\dotsc,k,1',\dotsc,k'\}$:
	\begin{equation}
		\label{eq:blocks}
		B_{i}=\left\{\big(\sum_{r=1}^{i-1}a_{r}\big)+1,\ldots,\sum_{r=1}^{i}a_{r},\big(\big(\sum_{r=1}^{i-1}b_{r}\big)+1\big)',
		\ldots,\big(\sum_{r=1}^{i}b_{r}\big)'\right\}.
	\end{equation}
	By convention, if $a_i=0$ (or $b_i=0$), then the corresponding enumeration of terms in $B_i$ is empty. Define the \defn{canonical partition diagram} $d_{\Gamma}$ associated to $[\Gamma]$ as           
	\begin{equation}
		\label{eq:part_connect}
		d_{\Gamma}=\{B_{1},\ldots,B_{s}\}.
	\end{equation}
	
Note that $d_\Gamma \in O_\Gamma$, and is independent of choice of representatives of $[\Gamma]\in\tilde{\mathcal{B}}_k$.

\textbf{Action of $S_k \times S_k$.} For  $(\sigma,\tau)\in S_{k}\times S_{k}$ and $d\in \mathcal{A}_k$, the permutations $\sigma$ and $\tau$ act by substitution on the primed and unprimed elements of $d$ respectively. Observe that $O_{\Gamma}$ is the $S_{k}\times S_{k}$-orbit of $d_{\Gamma}$. 
Let $(S_k\times S_k)\setminus \mathcal{A}_k$ denote the set of orbits of $S_k\times S_k$ on $\mathcal{A}_k$. 
		The following map is a bijection 
		\begin{align}\label{eq:bij}
			\psi:\tilde{\mathcal{B}}_k&\to (S_{k}\times S_{k})\setminus \mathcal{A}_k\\
			\psi([\Gamma])&=O_{\Gamma}.\nonumber
		\end{align}
	
	\begin{example}\label{ex:genericlabel}
		Let $\Gamma$ be the following graph:

			\begin{displaymath}
			\xymatrix{
				{}\ar@{}[d]|{\Gamma\quad=} &	\overset{0}{\bullet} \ar@{-}[d]\ar@{-}[dr] & \overset{1}{\bullet}  & \overset{2}{\bullet}  \ar@{-}[dl]\\
				{} &	\underset{0}{\bullet} & \underset{1}{\bullet} & \underset{2}{\bullet}\\
			}
		\end{displaymath}
		
		Then the set of edges of $\Gamma$ is $\{(0,0),(0,1),(2,1)\}$ in lexicographically increasing order. We ignore the $(0,0)$ edge in determining the blocks of $d_\Gamma$. By Equation~\eqref{eq:blocks}, the blocks are $B_{1}=\{1'\}$, $B_{2}=\{1,2,2'\}$, and so $d_{\Gamma}$ is as follows: 
		\begin{center}
			\begin{tikzpicture}[scale=0.9,mycirc/.style={circle,fill=black, minimum size=0.1mm, inner sep = 1.5pt}]
				\node[mycirc,label=above:{$1$}] (n1) at (0,1) {};
				\node[mycirc,label=above:{$2$}] (n2) at (1,1) {};
				\node[mycirc,label=below:{$1'$}] (n1') at (0,0) {};
				\node[mycirc,label=below:{$2'$}] (n2') at (1,0) {};
				\draw (n1)..controls(0.5,0.5)..(n2);
				\draw (n2)-- (n2');
			\end{tikzpicture}
		\end{center}

		Under the bijection \eqref{eq:bij}, $[\Gamma]$ is mapped to the set $O_\Gamma$, the orbit of $d_\Gamma$, consisting of the following partition diagrams:
		
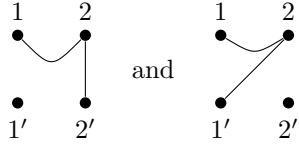
\begin{center}
			\begin{tikzpicture}[scale=0.9,mycirc/.style={circle,fill=black, minimum size=0.1mm, inner sep = 1.5pt}]
				\node[mycirc,label=above:{$1$}] (n1) at (0,1) {};
				\node[mycirc,label=above:{$2$}] (n2) at (1,1) {};
				\node[mycirc,label=below:{$1'$}] (n1') at (0,0) {};
				\node[mycirc,label=below:{$2'$}] (n2') at (1,0) {};
				\draw (n1)..controls(0.5,0.5)..(n2);
				\draw (n2)-- (n2');
				\node ( ) at (2,0.5) {and};
				\node[mycirc,label=above:{$1$}] (m1) at (3,1) {};
				\node[mycirc,label=above:{$2$}] (m2) at (4,1) {};
				\node[mycirc,label=below:{$1'$}] (m1') at (3,0) {};
				\node[mycirc,label=below:{$2'$}] (m2') at (4,0) {};
				\draw (m1)..controls(3.5,.7)..(m2);
				\draw (m2)-- (m1');
			\end{tikzpicture} 
		\end{center}
	\end{example}

	\begin{definition}\label{def:alpha}           
		Given $d\in \mathcal{A}_k$, let $d^l=\{B_1^{l},\dots,B_s^{l}\}$ (for the definition of $d^l$ see Definition~\ref{eg:par_alg_diag}), and $|B_i^l|=b_i$ for $i=1,\dotsc,s$. Define $\alpha_d$ to be the cardinality of $S_k$-orbit of $d^{l}$
		\begin{displaymath}
			\alpha_{d}=\frac{k!}{b_1!\cdots b_s!}.
		\end{displaymath} 
	\end{definition}
	
	Now we give an embedding of $\MP_k(\xi)$ inside $\Pa_k(\xi)$.
	\begin{theorem}\label{cor:new_lambda}
		The following map is an injective algebra homomorphism
		\begin{align}\label{al:eta-k}
			\tilde{\eta}_{k}:\MP_{k}(\xi)&\to \Pa_k(\xi)\\
			\tilde{\eta}_{k}([\Gamma])&=\frac{1}{\alpha_{d_\Gamma}}\sum_{d\in O_{\Gamma}} x_{d}\nonumber,
		\end{align}
		where $\{x_d\}$ is the orbit basis of $\Pa_k(\xi)$ as defined in Equation~\eqref{eq:xd}.
	\end{theorem}	
\textbf{Reduction of proof.} The injectivity of the map is clear using the bijection~\eqref{eq:bij}. 
		The structure constants for both $\MP_{k}(\xi)$ and $\Pa_k(\xi)$ are polynomials in $\xi$. Therefore it is sufficient to prove $\tilde{\eta}_{k}$ is an algebra homomorphism when 
		$\xi$ is evaluated at any positive integer $n$ such that $n\geqslant 2k$. For this, consider the map
		\begin{align*}
			\eta_{k}:\End_{S_n}(F[M(n,k)])) &\to \End_{S_n}(F[A(n,k)]) \\
			\eta_k(T_{[\Gamma]})&= \tau_k\circ T_{[\Gamma]}\circ\pi_k.
		\end{align*}
	Here recall that $\{T_{[\Gamma]}\mid [\Gamma]\in \tilde{\mathcal{B}}_{k,n}\}$ is a basis of $\End_{S_n}(F[M(n,k)])$, and we will define in Definition \ref{def:pro_sec}  the canonical projection map $\pi_k$ and   the canonical section map $\tau_k$ of $\pi_k$. Since  $\pi_k\circ\tau_k=id_{F[M(n,k)]}$, it follows that the map $\eta_k$ is an embedding of algebras. Now consider the following diagram
		\begin{displaymath}
			\xymatrix{
				\MP_{k}(n) \ar@{->>}[d]\ar@{^{(}->}[r]^{\tilde{\eta}_{k}} & \Pa_{k}(n)\ar@{->>}[d]\\
				\End_{S_n}(F[M(n,k)]) \ar@{^{(}->}[r]^{\eta_{k}} & \End_{S_n}(F[A(n,k)]),
			}
		\end{displaymath}                                                                                                                  
	where the leftmost and the rightmost vertical maps are as in Theorem~\ref{thm:swd_mult} and Theorem~\ref{thm:SWD_partition} respectively.  These vertical maps are algebra isomorphisms for $n\geqslant 2k$.  Therefore in order to show $\tilde{\eta}_k$ is a homomorphism of algebras, it is enough to show that the above diagram commutes
		\begin{equation}\label{eq:eta-lambda}
			\eta_{k}(T_{[\Gamma]})=\frac{1}{\alpha_{d_\Gamma}}\sum_{d\in O_{\Gamma}} x_{d},  \text{ for } [\Gamma]\in\tilde{\mathcal B}_{k,n}.      
		\end{equation}
This equation is the final reduction for the proof of the theorem. To prove it we need following definitions and notation.

	\begin{definition}
		\label{def:alpha_1} 
		For $\mathsf{a}=(a_1,\ldots,a_n)\in M(n,k)$, define  
		\begin{displaymath}
			D(\mathsf{a})=\{(A_1,\ldots, A_n) \mid \sqcup_{i=1}^{n}A_i =\{1,\ldots,k\},\, |A_i|=a_i\}.
		\end{displaymath}
In words,	the set $D(\mathsf{a})$ consists of ordered set partitions in $\mathcal{A}(n,k)$ whose $i$-th part has the cardinality equal to $a_i$.
For example if $\mathsf{a}=(1,1,0)\in M(3,2)$ then $D(\mathsf{a})= \{(\{1\},\{2\},\emptyset),(\{2\},\{1\},\emptyset)\}$. 
	\end{definition}

\begin{definition}
\label{def:pro_sec}
	The \defn{canonical projection map} $\pi_k:F[\mathcal{A}(n,k)]\to F[M(n,k)]$ and the \defn{canonical section map} $\tau_k:F[M(n,k)]\to F[\mathcal{A}(n,k)]$ are defined as
	\begin{align*}
		\pi_k(1_{(A_1,\ldots,A_n)})&= 1_{(|A_1|,\ldots, |A_n|)},\\
		\tau_k(1_{\mathsf{a}})&= \frac1{|D(\mathsf{a})|} \sum_{S \in D(\mathsf{a})}1_S.  
	\end{align*}   
\end{definition}                                                                                                                                  
For $S\in D(\mathsf{a})$, one has $\pi_k(1_{S})=1_{\mathsf{a}}$  and hence $\pi_k\circ \tau_k=id_{F[M(n,k)]}$.    
\begin{remark}
\label{rmk:classical maps}
Under the isomorphisms $\Sym^k(F^n) \cong F[M(n,k)]$ and $(F^n)^{\otimes k}\cong F[\mathcal{A}(n,k)]$, the map $\pi_k$  is the well-known projection map from $(F^n)^{\otimes k}$ to $\Sym^{k}(F^n)$ and $\tau_k$ is the corresponding well-known section map.
\end{remark}

The following example illustrates the idea of proof of Equation~\eqref{eq:diag_operator}. 
\begin{example}
Let $\Gamma$ be the graph as in	Example~\ref{ex:T}.   We have $k=2$ and $n=3$. 
We check the validity of Equation~\eqref{eq:eta-lambda}  on the basis of $F[\mathcal{A}(3,2)]$. 

 The left hand side of Equation~\eqref{eq:eta-lambda} at  the basis element $1_{(\{1,2\},\emptyset,\emptyset)}$ is given as follows
\begin{align*}
	\big(\tau_2\circ T_{[\Gamma]}\circ \pi_2\big)(1_{(\{1,2\},\emptyset,\emptyset)})
&=\tau_2\circ T_{[\Gamma]}(1_{(2,0,0)})\\
&=\tau_2(1_{(1,1,0)}+1_{(1,0,1)}), \text{ by Example~\ref{ex:T}}\\
&=\frac{1}{2}(1_{(\{1\},\{2\},\emptyset)}+ 1_{(\{2\},\{1\},\emptyset)}+1_{(\{1\},\emptyset,\{2\})}+1_{(\{2\},\emptyset,\{1\})}).
\end{align*}
Now we compute the right hand side of Equation~\eqref{eq:eta-lambda} at $1_{(\{1,2\},\emptyset,\emptyset)}$.  The set $O_{\Gamma}$, as given in Example~\ref{ex:genericlabel}, consists of two partition diagrams  $d_1=d_{\Gamma}=\{\{1,2,2'\},\{1'\},\emptyset\}$ and $d_2=\{\{1,2,1'\},\{2'\},\emptyset\}$. 
By Equation~\eqref{eq:diag_operator}, we have
\begin{align*}
x_{d_1}(1_{(\{1,2\},\emptyset,\emptyset)})&= 1_{(\{1\},\{2\},\emptyset)}	+ 1_{(\{1\},\emptyset, \{2\})}\\
x_{d_2}(1_{(\{1,2\},\emptyset,\emptyset)})&= 1_{(\{2\},\{1\},\emptyset)}	+ 1_{(\{2\},\emptyset, \{1\})}.
\end{align*}
Since $\alpha_{d_\Gamma}=2$,  both sides of Equation~\eqref{eq:eta-lambda} agree  at $1_{(\{1,2\},\emptyset,\emptyset)}$.

Computations for the elements $1_{(\emptyset, \{1,2\},\emptyset)}$ and $1_{(\emptyset, \emptyset, \{1,2\})}$  are similar. For every other basis element the left hand side of Equation~\eqref{eq:eta-lambda} vanishes as shown in Example~\ref{ex:T} and the right hand side of Equation~\eqref{eq:eta-lambda} vanishes due to  Equation~\eqref{eq:diag_operator}.
\end{example}

Now we are ready to prove Equation~\eqref{eq:eta-lambda} which completes the proof of Theorem~\ref{cor:new_lambda}.
	\begin{proof}\label{proof:reduction}
Let $\Gamma\in\mathcal{B}_{k}$ be a graph with its edges $(a_1,b_1),\dotsc,(a_n,b_n)$ arranged in lexicographically increasing order.
 Denote  the elements $(a_1,\ldots,a_n)$ and $(b_1,\ldots,b_n)$ in $M(n,k)$  by $ \mathsf{a}$ and $ \mathsf{b}$ respectively. 
 Let $A=(A_1,\ldots, A_n)\in \mathcal{A}(n,k)$ be an ordered set partition  of $\{1,\ldots,k\}$. 
 
 		The left hand side of Equation~\eqref{eq:eta-lambda} is
		\begin{align}\label{eq:rhs}
			\tau_k\circ T_{[\Gamma]} \circ \pi_k(1_A)
			&= \tau_k\big(T_{[\Gamma]}(1_{(|A_1|,\ldots,|A_n|)})\big), \text{ by Definition \ref{def:pro_sec}}\nonumber\\
			& =\tau_k\bigg(\sum_{\{\mathsf{c}\mid [\Gamma_{(|A_1|,\ldots,|A_n|),\mathsf{c}}]=[\Gamma]\}}1_{\mathsf{c}}\bigg), \text{ by Definition \ref{def:integral operator}}
		\end{align}
We have the following two cases.

		{\bf Case 1:} Assume $\set\big((|A_1|,\ldots,|A_n|)\big)\neq \set(\mathsf{a})$.  Then for every $\mathsf{c}\in M(n,k)$ we have
		\begin{displaymath}
			[\Gamma_{(|A_1|,\ldots,|A_n|),\mathsf{c}}]\neq [\Gamma].
		\end{displaymath}
		So  from Equation \eqref{eq:rhs} $\tau_k\circ T_{[\Gamma]}\circ \pi_k(1_{A})=0$. 
		For $d\in O_{\Gamma}$, the set of cardinalities of parts of $d^{u}$ is $\set(\mathsf{a})$ and therefore $d^{u}\neq \set(A)$. So by Equation~\eqref{eq:diag_operator}, the right hand side of Equation \eqref{eq:eta-lambda} is 
		$$\frac{1}{\alpha_{d_\Gamma}}\sum_{d\in O_{\Gamma}}x_d(1_A)=0.$$                                                                                                                                                    
		
	{\bf Case 2: }	Assume $\set\big((|A_1|,\ldots,|A_n|)\big)=\set(\mathsf{a})$.    Following Equation~\eqref{eq:rhs}, we have
			\begin{align*}
		\tau_k\bigg(\sum_{\{\mathsf{c}\mid [\Gamma_{(|A_1|,\ldots,|A_n|),\mathsf{c}}]=[\Gamma]\}}1_{\mathsf{c}}\bigg)&=\sum_{\{\mathsf{c}\mid [\Gamma_{(|A_1|,\ldots,|A_n|),\mathsf{c}}]=[\Gamma]\}} \frac{1}{|D(\mathsf{c})|}\sum_{S\in D(\mathsf{c})}1_S, \text{ by Definition \ref{def:pro_sec}}\nonumber\\
		&=\frac{1}{\alpha_{d_\Gamma}}\sum_{\{\mathsf{c}\mid [\Gamma_{(|A_1|,\ldots,|A_n|),\mathsf{c}}]=[\Gamma]\}} \sum_{S\in D(\mathsf{c})}1_S.
		\end{align*}
		For the last equality recall that for $\Gamma=\Gamma_{(|A_1|,\ldots,|A_n|),\mathsf{c}}$, we have $\alpha_{d_{\Gamma}}=|D(\mathsf{c})|$.        
		                                                                                                                                            
	For $d\in O_{\Gamma}$ and $d^u\neq \set(A)$ then from Equation \eqref{eq:diag_operator},
		$x_{d}(1_{A})=0.$ 	Thus we may restrict the sum in  the right hand side of Equation~\eqref{eq:eta-lambda}
		\begin{align*}
			\frac{1}{\alpha_{d_\Gamma}}\sum_{d \in O_{\Gamma}}x_d(1_A)&=
			\frac{1}{\alpha_{d_\Gamma}} \sum_{\substack{d\in O_{\Gamma} \\d^{u}=\set(A)}} \sum_{\substack{S=(S_1,\dots,S_n) \\ \{(A_1,S_1),\ldots,(A_n,S_n)\}=d} } 1_{S} \text{,\quad\quad by Equation~\eqref{eq:diag_operator}}\\
			&=\frac{1}{\alpha_{d_\Gamma}}  \sum_{\substack{S=(S_1,\dots,S_n) \\ \Gamma_{(|A_1|,\ldots, |A_n| ),(|S_1|,\ldots,|S_n|)}=\Gamma} } 1_{S} .
		\end{align*}
For the second equality, observe that the subset  $\{d\in O_{\Gamma} \mid d^{u}=\set(A)\}$ comprises exactly diagrams of the form $d=\{(A_1,S_1),\dotsc,(A_n,S_n)\}$ with $\Gamma_{\{(|A_1|,\dotsc,|A_n|),(|S_1|,\dotsc,|S_n|)\}}=\Gamma$. 

We may rearrange the order of the summations to first fix the cardinalities $\mathsf{c}=(c_1,\dotsc,c_n)$ such that $\Gamma_{\{(|A_1|,\dotsc,|A_n|),(c_1,\dotsc,c_n)\}}=\Gamma$, and then range over all ordered set partitions $S=(S_1,\dotsc,S_n)$ such that $|S_i|=c_i$ (for $i=1,\dotsc,n$):
		\begin{align}\label{eq:LHS}
							\frac{1}{\alpha_{d_\Gamma}}\sum_{d \in O_{\Gamma}}x_d(1_A)&=\frac{1}{\alpha_{d_\Gamma}} \sum_{\{\mathsf {c}| [\Gamma_{(|A_1|,\ldots,|A_n|),\mathsf{c}}]=[\Gamma]\}} \sum_{S \in D(\mathsf {c})}1_{S}.
		\end{align}
Comparing Equation \eqref{eq:LHS} and Equation \eqref{eq:rhs} we see that Equation \eqref{eq:eta-lambda} is true. 	\end{proof}

 \subsection{Cellularity of $\MP_k(\xi)$} 
	The notion of  cellular algebras was introduced in \cite{GL} by giving certain basis satisfying very specific properties analogous to  the Kazhdan--Lusztig basis of Hecke algebras. Later, an equivalent basis-free definition \cite[Definition 3.2]{KX} was given, which we recall in Definition \ref{def:cell}. We note that our proof of cellularity of $\MP_{k}(\xi)$ uses only Proposition \ref{prop:idem}.
	
	\begin{definition}\label{def:cell}
		Let $A$ be algebra over $F$. Suppose $\mathfrak{i}$ is an anti-involution on $A$. A two-sided ideal $J$ in $A$ is called a cell ideal if and only  if $\mathfrak{i}(J)=J$
		and there exists a left ideal $\Delta \subseteq J$ such that $\Delta$ is finitely generated  over $F$ and there is an isomorphism of $A$-bimodules such that $\alpha:J\cong \Delta\otimes_{F}\mathfrak{i}(\Delta)$  making the following diagram commutative:
		\begin{displaymath}
			\xymatrix{ 
				J\ar@{->}[d]^{\mathfrak{i}}\ar@{->}[r]^-{\alpha} & \Delta\otimes_{F}\mathfrak{i}(\Delta) \ar@{->}[d]^{x\otimes y\mapsto \mathfrak{i}(y)\otimes \mathfrak{i}(x)} \\
				J \ar@{->}[r]^-{\alpha} &\Delta\otimes_{F}\mathfrak{i}(\Delta)
			} 
		\end{displaymath}
		
		The algebra $A$, with respect to the anti-involution $\mathfrak{i}$, is called \defn{cellular} if and only if 
		\begin{itemize}
	\item	there is a finite chain 
		of two sided ideals of $A: 0=J_{0}\subseteq J_{1}\subseteq \cdots \subseteq J_n=A$, each of them are fixed by $\mathfrak{i}$,
		\item for each $1\leqslant l\leqslant n$, the quotient $J_{l}/J_{l-1}$ is a cell ideal of $A/J_{l-1}$ with respect to the anti-involution induced by $\mathfrak{i}$ on the quotient $A/J_{l-1}$.
		\end{itemize}
	\end{definition}
	In the definition of a cellular algebra $A$, cell ideals of $A$ lead to some special modules $\Delta$ called as cell modules (analogous to Specht modules in the case of symmetric group).  Further, cell modules enable one to construct all simple modules of $A$ (for detail see \cite{GL},\cite{KX}).  The following proposition~\cite[Proposition 4.3]{KX} allows one to realize a cellular subalgebra of $A$ with respect to a choice of  an idempotent in $A$.  
	
	\begin{proposition}\label{prop:idem}
		Let $A$ be a cellular algebra with respect to an anti-involution $\mathfrak{i}$. Let $e\in A$ be an idempotent such that $\mathfrak{i}(e)=e$.  
		Then the subalgebra $eAe$ is also cellular  with respect to the anti-involution $\mathfrak{i}$ restricted to $eAe$. 
	\end{proposition}
With the same hypothesis as in Proposition \ref{prop:idem}, if $J$ is a cell ideal and $\Delta$ is a cell module of $A$ , then $eJe$ is a cell ideal and $e\Delta$ is a cell module of  $eAe$.  (Utilizing these, Lemma \ref{lm:idem},  and Theorem \ref{thm:cellular} we hope to construct cell ideals and cell modules for $\MP_{k}(\xi)$ in the future.)

	\begin{definition}
	Let $\mathcal{Y}_{k}$ be the subset of $\mathcal{A}_{k}$ consisting of partition diagrams $d=\{B_1,B_2,\ldots, B_n\}$  satisfying the following condition: 
	\begin{equation}\label{eq:bal}
		|B_j^{u}|=|B_j^{l}| \text{ for $1\leqslant j\leqslant n$}.
	\end{equation} 
\end{definition}

	\begin{example} \label{ex:e}For $k=2$, the identity element of $\MP_k(\xi)$ is 
		
		\begin{displaymath}
			\xymatrix{ 
				{} \ar @{}[d]|{ id \quad =}&
				\overset{0}{\bullet} & \overset{1}{\bullet}\ar@2{-}[d] & \overset{2}{\bullet}\\
				{}& \underset{0}{\bullet} & \underset{1}{\bullet}  & \underset{2}{\bullet}
			}
			\quad
			\xymatrix{ 
				{} \ar @{}[d]|{ + }&
				\overset{0}{\bullet}  & \overset{1}{\bullet}&\overset{2}{\bullet}\ar@1{-}[d]\\
				{}&  \underset{0}{\bullet} &  \underset{1}{\bullet} &\underset{2}{\bullet}}	
		\end{displaymath}
		
	Also,	we have $\mathcal{Y}_2=\{d_1,d_2,d_3\}$, where 
	\begin{center}
			\begin{tikzpicture}[scale=0.9,mycirc/.style={circle,fill=black, minimum size=0.1mm, inner sep = 1.5pt}]
				\node (1) at (-1,0.5) {$d_1=$};
				\node[mycirc,label=above:{$1$}] (n1) at (0,1) {};
				\node[mycirc,label=above:{$2$}] (n2) at (1,1) {};
				\node[mycirc,label=below:{$1'$}] (n1') at (0,0) {};
				\node[mycirc,label=below:{$2'$}] (n2') at (1,0) {};
				\draw (n1)--(n1');
				\draw (n2)-- (n2');
			\end{tikzpicture}
			\begin{tikzpicture}[scale=0.9,mycirc/.style={circle,fill=black, minimum size=0.1mm, inner sep = 1.5pt}]
				\node (1) at (-1,0.5) {, $d_2=$};
				\node[mycirc,label=above:{$1$}] (n1) at (0,1) {};
				\node[mycirc,label=above:{$2$}] (n2) at (1,1) {};
				\node[mycirc,label=below:{$1'$}] (n1') at (0,0) {};
				\node[mycirc,label=below:{$2'$}] (n2') at (1,0) {};
				\draw (n1)--(n2');
				\draw (n2)-- (n1');
			\end{tikzpicture}
			\begin{tikzpicture}[scale=0.9,mycirc/.style={circle,fill=black, minimum size=0.1mm, inner sep = 1.5pt}]
				\node (1) at (-1,0.5) {, $d_3=$};
				\node[mycirc,label=above:{$1$}] (n1) at (0,1) {};
				\node[mycirc,label=above:{$2$}] (n2) at (1,1) {};
				\node[mycirc,label=below:{$1'$}] (n1') at (0,0) {};
				\node[mycirc,label=below:{$2'$}] (n2') at (1,0) {};
				\draw (n1)..controls(0.5,0.5)..(n2);
				\draw (n2)-- (n2');
				\draw (n1)--(n1');
			\end{tikzpicture}
		\end{center}

		One can check that $\tilde{\eta}_{k}({id})=\frac{1}{2}(x_{d_1}+x_{d_{2}})+x_{d_{3}}$, where
		$\tilde{\eta}_k$ is defined in Equation~\eqref{al:eta-k}.
		
	\end{example}

	\begin{lemma}\label{lm:idem}
		Let $e = \sum_{d\in\mathcal{Y}_k} \frac{1}{\alpha_{d}} x_{d}$.
		Then $e$ is an idempotent and 
		the embedding $\tilde{\eta}_{k}$ induces an isomorphism between
		$\MP_{k}(\xi)$ and $e\Pa_{k}(\xi)e$.
	\end{lemma} 
	
	\begin{proof}
		We show $\tilde{\eta}_k(id)=e$, where $id=\sum_{[\Gamma]\in\mathcal{U}_k}[\Gamma]$ is the identity element of $\MP_k(\xi)$ (see Equation \eqref{mpa:id}). From Equation~\eqref{al:eta-k}
		\begin{displaymath}
			\tilde{\eta}_{k}(id)=\sum_{[\Gamma]\in\mathcal{U}_k}\frac{1}{\alpha_{d_{\Gamma}}} \sum_{d\in O_{\Gamma}}x_{d}.
		\end{displaymath}
		For $[\Gamma]\in \mathcal{U}_k$, by construction of $d_{\Gamma}$ we have $d_{\Gamma}\in \mathcal{Y}_k$. Also, since $\mathcal{Y}_k$ is closed under the action of $S_k\times S_k$,  the orbit $O_{\Gamma}\subseteq \mathcal{Y}_k$. Conversely, given $d\in \mathcal{Y}_k$, we have $[G_d]\in \mathcal{U}_k$  (see Section \ref{subsec:embed} for the definition of $G_d$). So we obtain $\bigsqcup_{[\Gamma]\in \mathcal{U}_k}O_{\Gamma}=\mathcal{Y}_k$. It follows that $\tilde{\eta}_k(id)=e$ by observing that for $d_1,d_2$ in the same orbit,
		$\alpha_{d_1}=\alpha_{d_2}$.
		 The map $\tilde{\eta}_k$ is an algebra homomorphism and $\tilde{\eta}_{k}(id)=e$, therefore $e$ is an idempotent.
		 
Also, the image of $\tilde{\eta}_k$ is contained in 
		$e\Pa_k(\xi)e$. For $d_0\in \mathcal{A}_k$, we show that $ex_{d_{0}}e$ is in the image of $\tilde{\eta}_k$. In particular, we show that $ex_{d_{0}}e$ is a scalar multiple of $\tilde{\eta}_k([\Gamma])$ for $\Gamma=G_{d_{0}}$. We have
		\begin{align}\label{eq:image_idem}
			ex_{d_{0}}e&=\sum_{d_1,d_2\in\mathcal{Y}_k}\frac{1}{\alpha_{d_1}\alpha_{d_2}}x_{d_{1}}x_{d_0}x_{d_2} \nonumber\\&= \sum_{\substack{d_1,d_2\in\mathcal{Y}_k\\ d_2^l=d_0^u}}\frac{1}{\alpha_{d_{1}}\alpha_{d_{2}}}x_{d_1}x_{d_{0}\circ {d_{2}}} \nonumber\\
			&=\sum_{\substack{d_1,d_2\in\mathcal{Y}_k \\ d_2^l=d_0^u \\d_0^l=d_1^u}}\frac{1}{\alpha_{d_{1}}\alpha_{d_{2}}}x_{d_{1}\circ d_{0}\circ d_{2}}.
		\end{align}
The first equality in the above is due to Theorem \ref{thm:xd} : since $d_2\in \mathcal{Y}_k$ (in particular every vertex in the top row is connected to a vertex in the bottom row), $[d_0\circ d_2 ]=0$ and there is no coarsening of $d_0\circ d_2$. Likewise the second equality holds with the extra observation that if $d_2^l=d_0^u$ then  $(d_0\circ d_2)^l=d_0^l$. 
		
		Since $d_1\in\mathcal{Y}_k$ and $d_0^l=d_1^u$, we have $\alpha_{d_0}=\alpha_{d_{1}}$. For $d_1,d_2\in\mathcal{Y}_k$ with $d_2^l=d_0^u$ and $d_0^l=d_1^u$, we see that $d_0$ and $d_1\circ d_0\circ d_2$ are in the same $S_k\times S_k$-orbit $O_{\Gamma}$ for $\Gamma=G_{d_{0}}$. Conversely, if $d$ in $S_k\times S_k$-orbit $O_{\Gamma}$, then there exist $d_1,d_2\in \mathcal{Y}_k$ such that $d=d_1\circ d_0 \circ d_2$.  
		
		If $d_0$ has exactly $i_{(a,b)}$ blocks $B_{m_1},\ldots, B_{m_{i_{(a,b)}}}$ ($a, b$ are non-negative integers both simultaneously not zero) such that 
		$|B_{m_s}^{u}|=a$ and $|B_{m_s}^l|=b$ for $1\leqslant s\leqslant i_{(a,b)}$ then there are  $r_{d_{0}}=\prod_{(a,b)} (i_{(a,b)}!)$ pairs of partition diagrams $(d_1,d_2)$ in $\mathcal{Y}_k\times\mathcal{Y}_k$ such that $d_1\circ d_0 \circ d_2=d$. Since $d_2^l=d_0^u$ therefore 
		$\alpha_{d_2}$ is the cardinality of $S_k$-orbit of $d_0^u$. Denote this cardinality by $\beta_{d_0}$. Then the sum~\eqref{eq:image_idem} simplifies to 
		\begin{align*}
			\sum_{d\in O_{\Gamma}}\frac{r_{d_{0}}}{\alpha_{d_0}\beta_{d_{0}}}x_{d}&=\frac{r_{d_0}}{\beta_{d_0}}\bigg(\frac{1}{\alpha_{d_{0}}}\sum_{d\in O_{\Gamma}} x_{d}\bigg)=\frac{r_{d_0}}{\beta_{d_0}}\tilde{\eta}_k([\Gamma])
		\end{align*}
		where $[\Gamma]\in \tilde{\mathcal{B}}_k$ corresponds to $O_{\Gamma}$ under the bijection~\eqref{eq:bij}. Thus we see that $ex_{d_{0}}e$ is in the image of $\tilde{\eta}_k$.
	\end{proof}

	In the following theorem we show that the multiset partition algebras are generically semisimple over $F$. That is, $\MP_{k}(v)$ is semisimple over $F$ for all but finitely many values 
$v$	in $F$.
	
\begin{theorem}\label{thm:gensem}
The algebra $\MP_k(\xi)$ is semisimple over $F[\xi]$.  Furthermore, for $v\in F$,  $\MP_{k}(v)$ over $F$ is semisimple when $v$ is not an integer or when $v$ is an integer such that $v\notin\{0,1,\ldots, 2k-2\}$.
	\end{theorem}
\begin{proof}
	Since $e$ is an idempotent, the radical rad$(e\Pa_k(\xi)e)=e(\text{rad}(\Pa_k(\xi)))e$. We know that $\Pa_{k}(\xi)$ is semisimple over $F[\xi]$ and also semisimple when $v$ is not an integer or $v$ is a integer such that $v\notin\{0,1,\ldots,2k-2\}$. In all these cases, so the radical of the partition algebra is trivial.  Now by Lemma~\ref{lm:idem} we have that the radical of the multiset partition algebra is also trivial in these cases, in particular it is semisimple.
 \end{proof}
 
\begin{definition}
	The anti-involution $\mathfrak{i}$ is defined for $d=\{(B_1^u,B_1^l),\ldots,(B_n^u,B_n^l)\}\in \mathcal{A}_k$ by:
	\begin{equation*}
		\mathfrak{i}(d)=\{(B_1^l,B_1^u),\ldots,(B_{n}^l,B_{n}^u)\}
	\end{equation*}
	which may be visualized as interchanging each primed element $j'$ with unprimed
	element $j$ and vice versa. This map is extended linearly to $\Pa_k(\xi)$.
	\end{definition}
	
	 Xi~\cite{Xi} showed that partition algebras are cellular with respect to $\mathfrak{i}$. Under the embedding in Theorem \ref{cor:new_lambda}, the anti-involution $\mathfrak{i}$ on $\Pa_{k}(\xi)$ translates to the anti-involution $\tilde{\mathfrak{i}}$ on $\MP_{k}(\xi)$. For $[\Gamma]\in\tilde{\mathcal{B}}_k$, $\tilde{\mathfrak{i}}([\Gamma])$ is obtained by reflecting $\Gamma$ along the horizontal axis. 

	\begin{theorem}\label{thm:cellular}
	For $v\in F$,	 $\MP_{k}(v)$ is cellular with respect to the anti-involution $\tilde{\mathfrak{i}}$. 
	\end{theorem}
	
	\begin{proof}
		From Lemma~\ref{lm:idem} we have $\MP_{k}(v)\cong e \Pa_{k}(v) e$, where
		\begin{displaymath}
			e=\sum_{d\in\mathcal{Y}_k}\frac{1}{\alpha_d}x_d.
		\end{displaymath}
		We show that $e\Pa_k(v)e$ is a cellular algebra with respect to the anti-involution $\mathfrak{i}$. 
		The restriction of the map $\mathfrak{i}$ to $\mathcal{Y}_k$ is a bijection on $\mathcal{Y}_k$. Moreover, for $d\in\mathcal{Y}_k$, we have $\mathfrak{i}(x_{d})=x_{\mathfrak{i}(d)}$. So $\mathfrak{i}(e)=\sum_{d\in \mathcal{Y}_k}\frac{1}{\alpha_{d}}x_{\mathfrak{i}(d)}$. By observing that for $d\in\mathcal{Y}_k$, the cardinality of $S_k$-orbit of $d^u$ is the same as the cardinality of $S_k$-orbit of $d^{l}$, i.e, $\alpha_{d}=\alpha_{\mathfrak{i}(d)}$. We conclude that $\mathfrak{i}(e)=e$. By Proposition~\ref{prop:idem} we have $e\Pa_k(v)e$ is a cellular algebra.
	\end{proof}                                                                                                                                                         
	
	\section{The generalized multiset partition algebra}
	\label{app:A}
	For every vector $\lambda=(\lambda_1,\ldots, \lambda_s)$ of non-negative integers, 	let $\mathcal{V}_{\lambda}$ be the set of vectors
	\begin{equation}
		\mathcal{V}_\lambda=\{(i_1,i_2,\ldots,i_s)\in\ZZ_{\geqslant 0}^{s}\mid i_j\leqslant \lambda_j,\,1\leqslant j\leqslant s\}. 
	\end{equation}
	Let $\Gamma$ be a {multigraph} whose vertices are arranged in two rows. The vertices in the top row are labelled by the elements of  $\mathcal{V}_\lambda$, as are the vertices in the bottom row. Let $E_{\Gamma}$ denote the multiset of edges of $\Gamma$. Every edge of $\Gamma$ connects a vertex in the top row to a vertex in the bottom row. 
	The \defn{weight} of an edge $e$ of $\Gamma$ joining $(i_1,i_2,\ldots,i_s)$ in the top row with $(j_1,j_2,\ldots,j_s)$ in the bottom row is 
	$((i_1,i_2,\ldots,i_s),(j_1,j_2,\ldots,j_s))\in \ZZ_{\geqslant 0}^s\times \ZZ_{\geqslant 0}^s$. The \defn{weight} of $\Gamma$ is the (pointwise) sum of the weights of its edges. An edge is said to be \defn{non-zero} if its weight is not $({\bf 0},{\bf 0})$, with ${\bf 0}:=(0,\dotsc,0))$. The \defn{rank} of $\Gamma$, denoted $\rn(\Gamma)$, is the number of non-zero edges of $\Gamma$.

Let $\mathcal{B}_{\lambda}$ denote the set of multigraphs of weight $(\lambda,\lambda):=((\lambda_1,\ldots,\lambda_s),(\lambda_1,\ldots,\lambda_s))$. 
Two multigraphs in $\mathcal{B}_{\lambda}$ are said to be equivalent if they have the same non-zero edges. Let $\tilde{\mathcal{B}}_{\lambda}$ denote the set of equivalence classes in $\mathcal{B}_{\lambda}$. We define
	$\rn ({[\Gamma]}):=\rn(\Gamma)$. 
	
	Let $(I,J):=((i_1,i_2,\ldots,i_s),(j_1,j_2,\ldots,j_s))$ be a non-zero weighted edge of a multigraph in $\mathcal{B}_{\lambda}$. Then the following correspondence 
	\begin{equation}\label{mp cor}
		(I,J)\leftrightarrow\{1^{i_1},\ldots,s^{i_s},1'^{j_1},\ldots,s'^{j_s}\} 
	\end{equation}
	extends to a bijection between 
	the set $\tilde{\mathcal{B}}_{\lambda}$ and the set of multiset partitions of  
	$\{1^{\lambda_1},\ldots,s^{\lambda_s},1'^{\lambda_1},\ldots,s'^{\lambda_s}\}$. Thus the elements of $\tilde{\mathcal{B}}_{\lambda}$ depict the multiset partitions of $\{1^{\lambda_1},\ldots,s^{\lambda_s},1'^{\lambda_1},\ldots,s'^{\lambda_s}\}$.
	
	\begin{example}
		Consider the following graph in $\mathcal{B}_{(2,1)}$.
		
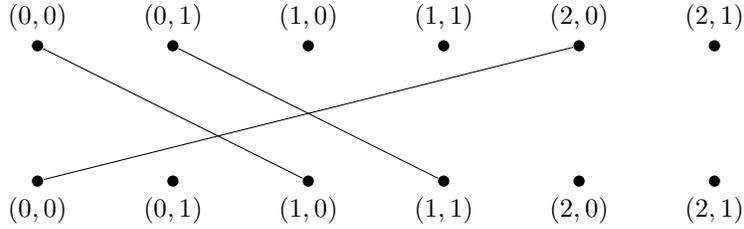
\begin{displaymath}
			\xymatrix{
			\overset{(0,0)}{\bullet}\ar@{-}[rrd] & \overset{(0,1)}{\bullet}\ar@{-}[rrd] & \overset{(1,0)}{\bullet} & \overset{(1,1)}{\bullet} & \overset{(2,0)}{\bullet}\ar@{-}[lllld] & \overset{(2,1)}\bullet\\
			 \underset{(0,0)}{\bullet} & \underset{(0,1)}{\bullet}& \underset{(1,0)}{\bullet} & \underset{(1,1)}{\bullet} & \underset{(2,0)}{\bullet} & \underset{(2,1)}\bullet }
			\end{displaymath}
			The multiset partition of $\{1^2,2,1'^2,2'\}$ associated to this graph is $$\{\{1'\},\{2,1',2'\},\{1^2\}\}.$$
	\end{example}

	We define $\MP_{\lambda}(\xi)$ to be the free module over $F[\xi]$ with basis $\tilde{\mathcal{B}}_{\lambda}$, and let $\mathcal{U}_{\lambda}$ be the subset of $\tilde{\mathcal{B}}_{\lambda}$ consisting of equivalence classes of graphs whose edges are of the form $(I,I)$ for $I\in\mathcal{V}_{\lambda}$. We will outline the minor modifications to Section \ref{sec:mpa} that are required to specify the structure constants with respect to the basis $\tilde{\mathcal{B}}_{\lambda}$.

Given $\Gamma,\Gamma_1,\Gamma_2\in\mathcal{B}_\lambda$ with $n$ edges each,  we can define \defn{configuration of paths}, and the \defn{support} $\Supp^n(\Gamma_1,\Gamma_2)$,  and the sets $\Supp^{n}_{\Gamma}(\Gamma_1,\Gamma_2)$ and $C_{\Gamma_1\Gamma_2}^{\Gamma}(n)$ as in Section \ref{sec:mpa}. For a configuration $P$, $\Gamma_P$ is defined to be the graph obtained by ignoring the middle vertex in each path in $P$. 
	
	For an edge of weight $(I,J)=((i_1,i_2,\ldots,i_s),(j_1,j_2,\ldots,j_s))$ in $\Gamma_P $ with $P\in \Supp^n_{\Gamma}(\Gamma_1,\Gamma_2)$, and any vertex $L\in \mathcal{V}_{\lambda}$,
	define
	\begin{displaymath}
		p_{IJ}(L):=\text{number of times the path }(I,L,J) \text{ occurs in } P.
	\end{displaymath}
	The multiplicity of the edge $(I,J)$ in $\Gamma_P$ is $p_{IJ}=\sum_{L\in \mathcal{V}_{\lambda}} p_{IJ}(L)$. 
	
	Let $v_0,v_1,\ldots,v_a$ be an enumeration of the elements of $\mathcal{V}_{\lambda}$ in the lexicographically increasing order (so in particular $v_0={\bf 0}$). For vector $\lambda=(\lambda_1,\ldots,\lambda_s)$ of non-negative integers, let $|\lambda|=\sum_{i=1}^{s} \lambda_{i}$.	Let $[\Gamma_1],[\Gamma_2],$ and $[\Gamma]$ be in $\tilde{\mathcal{B}}_{\lambda}$.   	The \defn{structure constant} of $[\Gamma]$ in the product $[\Gamma_1]*[\Gamma_2]$ is the following:
	\begin{align}\label{eq:poly_gen}
		\Phi_{[\Gamma_1][\Gamma_2]}^{[\Gamma]}(\xi)=
		\sum_{P \in \Supp^{3|\lambda|}_{\Gamma}(\Gamma_1,\Gamma_2)}\mathcal{K}_P\cdot(\xi-\rn{(\Gamma)})_{[P^{\Gamma_1\circ \Gamma_2}]},
	\end{align}
	where
	$[P^{\Gamma_1\circ \Gamma_2}]=\sum_{i=1}^{a}p_{\bf 0\bf 0}(v_i)$, 
	\begin{equation*}
		\mathcal{K}_P=  
		\frac{1}{p_{{\bf 0}{\bf 0}}(v_1)!\cdots p_{{\bf 0}{\bf 0}}(v_a)!}
		\prod_{(I,J)\in D_\Gamma\setminus\{({\bf 0},{\bf 0})\}}{p_{IJ}\choose p_{IJ}(v_0),\ldots,p_{IJ}(v_a)}
	\end{equation*}
	and $D_\Gamma$ denotes the set of all distinct edges of $\Gamma$.

	\subsection{Schur--Weyl duality} 
	In this section we state the Schur--Weyl duality between the actions of 
	$S_n$ and $\MP_{\lambda}(n)$ on $\Sym^{\lambda}(F^n)$.

	Recall that $\Sym^{\lambda}(F^n):=\Sym^{\lambda_1}(F^n)\otimes \Sym^{\lambda_2}(F^n)\otimes\cdots\otimes \Sym^{\lambda_s}(F^n)$.
	The choice of indexing set $M(n,\lambda_j)$ for a basis of each $\Sym^{\lambda_j}(F^n)$ yields the following indexing set $M(n,\lambda)$ for a basis of $\Sym^{\lambda}(F^n)$
	
	\begin{displaymath}
		M(n,\lambda):=\left\lbrace A=\begin{bmatrix}
			a_{11}       & a_{12} & \dots & a_{1n} \\
			\vdots       & \vdots & \dots & \vdots  \\
			a_{s1}       & a_{s2} & \dots & a_{sn}
		\end{bmatrix} \mid a_{ij}\in\ZZ_ {\geqslant 0}, \sum_{j=1}^{n}a_{ij}=\lambda_i \right\rbrace.
	\end{displaymath}
	The space $\Sym^{\lambda}(F^n)$ has a basis $\{e^A\mid A \in M(n,\lambda)\}$,
	where 
	\begin{displaymath}
		e^{A}:=e_1^{a_{11}} \dots e_n^{a_{1n}} \otimes e_1^{a_{21}} \dots e_n^{a_{2n}} \otimes \dots \otimes e_1^{a_{s1}} \dots e_n^{a_{sn}}.
	\end{displaymath}
	The symmetric group $S_n$ acts on an element $A\in M(n,\lambda)$ by permuting columns of $A$. The following isomorphism makes 
	$\Sym^\lambda (F^n)$ a permutation representation of $S_n$
	\begin{equation}\label{perm lambda}
		\Sym^\lambda (F^n) \cong F[M(n,\lambda)].
	\end{equation}
	
	As in Equation~\eqref{k-orbit-diagram bij}, the set of $S_n$-orbits of $M(n,\lambda)\times M(n,\lambda)$ is in bijection with the set $\tilde{\mathcal{B}}_{\lambda,n}$ consisting of elements of rank at most $n$ in $\tilde{\mathcal{B}}_{\lambda}$. In the following definition, we give the integral operator corresponding to an $S_n$-orbit of $M(n,\lambda)\times M(n,\lambda)$.

	\begin{definition} \label{integral operator lambda}
		For each $[\Gamma] \in \tilde{\mathcal B}_{\lambda,n} $   define $T_{[\Gamma]} \in \End_{S_n}(F[M(n,\lambda)])$ by 
		$$ T_{[\Gamma]} (1_{A})=\sum_{\{B \in M(n,\lambda) \mid[\Gamma_{A,B}]=[\Gamma]\}}1_{B}, \quad \text{for }A \in M(n,\lambda),$$
where the edges of $\Gamma_{A,B}$ connect a column of $A$ above to the corresponding column of $B$ below. 
	\end{definition}

	\begin{proposition}\label{thm:genT}
		The set
		$\{T_{[\Gamma]} \mid [\Gamma] \in \tilde{\mathcal B}_{\lambda,n}  \}$ is a basis of $\End_{S_n}(F[M(n,\lambda)])$. Given  $[\Gamma_1],[\Gamma_2] \in \tilde{\mathcal{B}}_{\lambda,n}$,  the product is given by
		$$ T_{[\Gamma_1]} T_{[\Gamma_2]}=\sum_{[\Gamma] \in \tilde{ \mathcal{B}}_{\lambda,n}} |C_{\Gamma _1 \Gamma _2}^{\Gamma}(n)|T_{[\Gamma]},
		$$
		where $\Gamma,\Gamma_1,\Gamma_2$ are the representative of the respective classes with $n$ edges.
	\end{proposition}
The proof of  Proposition \ref{thm:genT} is analogous to Proposition~\ref{thm:strcu_const}.

	\begin{theorem}\phantomsection\label{thm:swd}
		
		\begin{enumerate}
			\item Define a map 
			\begin{equation}
				\label{eq:swd lambda}
				\phi: \MP_\lambda(n) \to   \End_{S_n}(F[M(n,\lambda)])  \text{ by } 
			\end{equation}
			\[\phi([\Gamma])=
			\begin{cases}
				T_{[\Gamma]} &\text{ if $\rn[\Gamma]\leqslant n$},\\
				0 &\text{otherwise}.
			\end{cases}
			\]
			Then the map $\phi$ is a surjective  algebra homomorphism with the kernel
			$$\ker (\phi)=F\text{-$\spn$}\{[\Gamma]\in\tilde{\mathcal{B}}_{\lambda}\mid \rn([\Gamma])>n\}.$$ In particular, when $n\geqslant 2|\lambda|$, 
			$\MP_{\lambda}(n) \cong \End_{S_n}(F[M(n,\lambda)])$.
			\item The group algebra $F[S_n]$ generates $\End_{\MP_{\lambda}(n)}(F[M(n,\lambda)])$.
		\end{enumerate}
	\end{theorem}
		The proof of  Theorem \ref{thm:swd} is analogous to Theorem \ref{thm:swd_mult}.

Our next theorem is the analogue of Theorem \ref{thm:mult_k} for $\MP_{\lambda}(\xi)$. 
	\begin{theorem}\label{thm:multi_rule}
		For $[\Gamma_1],[\Gamma_2]$ in $\tilde{\mathcal{B}}_\lambda$,
		we define the following operation:
		\begin{equation}\label{multiplication rule}
			[\Gamma_{1}]*[\Gamma_2]=\sum_{[\Gamma] \in \tilde{\mathcal{B}}_{\lambda}} \Phi_{[\Gamma_1][\Gamma_2]}^{[\Gamma]}(\xi)[\Gamma],
		\end{equation}
		where $\Phi_{[\Gamma_1][\Gamma_2]}^{[\Gamma]}(\xi)\in F[\xi]$ is given in Equation~\eqref{eq:poly_gen}. 
		The linear extension of the operation~\eqref{multiplication rule} makes $\MP_\lambda(\xi)$ an associative unital algebra  over $F[\xi]$ with the identity element $id=\sum_{[\Gamma]\in\mathcal{U}_\lambda}[\Gamma]$.
	\end{theorem}

	\subsection{Cellularity of $\MP_{\lambda}(\xi)$}
	In this section, we give an embedding of $\MP_{\lambda}(\xi)$ inside $\Pa_{|\lambda|}(\xi)$ by generalizing the bijection in Equation~\eqref{eq:bij}. 
	
	Let $\lambda=(\lambda_{1},\lambda_{2},\ldots,\lambda_{s})$ be a vector of non-negative integers such that $|\lambda|=k$. Given $\Gamma\in\mathcal{B}_{\lambda}$ with $n$ edges (not necessarily non-zero), we have a pair
	$(P,Q)$ of $s\times n$ matrices in $M(n,\lambda)$ that uniquely determines the graph. Explicitly, let
	\begin{displaymath}
		P=
		\begin{bmatrix}
			p_{11}       & p_{12} & \dots & p_{1n} \\
			p_{21}       & p_{22} & \dots & p_{2n} \\
			\vdots       & \vdots & \dots & \vdots  \\
			p_{s1}       & p_{s2} & \dots & p_{sn}
		\end{bmatrix} \text{ and }
		Q=
		\begin{bmatrix}
			q_{11}       & q_{12} & \dots & q_{1n} \\
			q_{21}       & q_{22} & \dots & q_{2n}\\
			\vdots       & \vdots & \dots & \vdots  \\
			q_{s1}       & q_{s2} & \dots & q_{sn}
		\end{bmatrix}.
	\end{displaymath}

	Each column of $P$ is the label for a vertex in the top row of $\Gamma$ and each column of $Q$ is the label for a vertex in the bottom row of $\Gamma$. A column of $P$ is connected by an edge in $\Gamma$ to the corresponding column of $Q$. The columns of $P$ and $Q$ are simultaneously permuted such that columns of $P$ are in lexicographically increasing  order, and if $(p_{1j},\ldots,p_{sj})^{tr}=(p_{1j+1},\ldots,p_{sj+1})^{tr}$ then $(q_{1j},\ldots,q_{sj})^{tr} \leqslant (q_{1j+1},\ldots,q_{sj+1})^{tr}$, where $tr$ denotes the transpose of a matrix. The graph so obtained is denoted $\Gamma_{P,Q}$.
	
	We define the \defn{canonical  partition diagram}  $d_{\Gamma} = \{B_{1},\ldots, B_{n}\}$ (some of $B_i$ could be empty)  of $\{1,\ldots,k,1',\ldots,k'\}$ corresponding to  $\Gamma=\Gamma_{P,Q}$, where
	\begin{align}
		B_{j}&= B_{1j} \cup \cdots \cup B_{sj} \nonumber  \quad\text{ and }\label{eq:blocks_lam_2}\\
		B_{ij}&:=\bigg\{\big(\sum_{r=1}^{i-1}\lambda_{r}+ \sum_{l=1}^{j-1}p_{il}\big)+1,\ldots, \big(\sum_{r=1}^{i-1}\lambda_{r}+ \sum_{l=1}^{j}p_{il}\big), \\
	&	\quad \quad \quad \quad \big(\sum_{r=1}^{i-1}\lambda_{r}+ \sum_{l=1}^{j-1}q_{il}+1\big)',\ldots, \big(\sum_{r=1}^{i-1}\lambda_{r}+ \sum_{l=1}^{j}q_{il}\big)'
		\bigg\}.\nonumber
	\end{align}
	It follows from the construction that if two graphs in $\mathcal{B}_{\lambda}$ are equivalent, then their corresponding canonical partition diagrams defined as above are the same.

	\begin{example}
		Let $\lambda=(2,1)$. Here 
		\begin{displaymath}
			\mathcal{V}_{\lambda}=\{(0,0),(0,1),(1,0),(1,1),(2,0),(2,1)\}.
		\end{displaymath}
		Let $\Gamma$ be the following multigraph in $\mathcal{B}_{\lambda}$:
			\begin{displaymath}
			\xymatrix{
				\overset{(0,0)}{\bullet}\ar@{-}[rrd] & \overset{(0,1)}{\bullet}\ar@{-}[rrd] & \overset{(1,0)}{\bullet} & \overset{(1,1)}{\bullet} & \overset{(2,0)}{\bullet}\ar@{-}[lllld] & \overset{(2,1)}\bullet\\
				\underset{(0,0)}{\bullet} & \underset{(0,1)}{\bullet}& \underset{(1,0)}{\bullet} & \underset{(1,1)}{\bullet} & \underset{(2,0)}{\bullet} & \underset{(2,1)}\bullet }
		\end{displaymath}

		Then corresponding pair of matrices are follows:
		\begin{displaymath}
			P=\begin{bmatrix}
				0 & 0 & 2\\
				0 & 1 & 0
			\end{bmatrix} \text{ and }
			Q=
			\begin{bmatrix}
				1 & 1 & 0\\
				0 & 1 & 0
			\end{bmatrix}.
		\end{displaymath}
		And we have:
		\begin{center}
			\begin{tabular}{ c c c }
				$B_{11}= \{1'\},$ & $B_{12}= \{2'\}$, & $B_{13}= \{1,2\},$\\
				$B_{21}= \emptyset,$ & $B_{22}= \{3,3'\},$ & $B_{23}= \emptyset$.\\
			\end{tabular}
		\end{center}
		Following Equation \eqref{eq:blocks_lam_2}, the blocks of  canonical set partition associated to $\Gamma$ are:
		\begin{displaymath}
			B_{1}=\{1'\}, B_{2}=\{3,2',3'\},B_{3}=\{1,2\}.
		\end{displaymath}
		This canonical partition diagram $d_{\Gamma}$ in $A_{3}$ is depicted as follows:
		
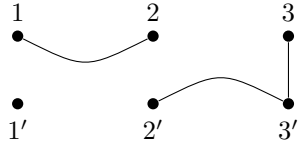
\begin{center}
			\begin{tikzpicture}[scale=0.9,
				mycirc/.style={circle,fill=black, minimum size=0.1mm, inner sep = 1.5pt}]
				\node[mycirc,label=above:{$1$}] (n1) at (0,2) {};
				\node[mycirc,label=above:{$2$}] (n2) at (2,2) {};
				\node[mycirc,label=above:{$3$}] (n3) at (4,2) {};
				\node[mycirc,label=below:{$1'$}] (n1') at (0,1) {};
				\node[mycirc,label=below:{$2'$}] (n2') at (2,1) {};
				\node[mycirc,label=below:{$3'$}] (n3') at (4,1) {};
				\draw (n1)..controls(1,1.5)..(n2);
				\draw (n2')..controls(3,1.5)..(n3');
				\draw (n3)--(n3');
			\end{tikzpicture}
		\end{center}
		
	\end{example}
 Let $S_\lambda$ denote the Young subgroup of $S_{\lvert \lambda \rvert}$. 	The  subgroup $S_{\lambda}\times S_{\lambda}$ of $S_{|\lambda|}\times S_{|\lambda|}$ acts on $\mathcal{A}_{|\lambda|}$ by the restriction. Let $(S_\lambda\times S_\lambda)\setminus \mathcal{A}_{|\lambda|}$ be the set of orbits of the action of $S_\lambda\times S_\lambda$ on $\mathcal{A}_{|\lambda|}$.  So analogously to Equation \eqref{eq:bij} we have the following  bijection
		\begin{align}\label{eq:bij lambda}
			\psi:\tilde{\mathcal{B}}_{\lambda}&\to (S_\lambda\times S_\lambda)\setminus \mathcal{A}_{|\lambda|}\\
			\psi([\Gamma])&=O_{\Gamma},\nonumber
		\end{align}
		where $[\Gamma]\in\tilde{\mathcal{B}}_{\lambda}$ and $O_{\Gamma}$ is the $S_\lambda\times S_\lambda$-orbit of $d_{\Gamma}$.

	For $\lambda=(1^k)$, the subgroup $S_\lambda\times S_\lambda$ is trivial and so it follows from Equation~\eqref{eq:bij lambda} the set $\tilde{\mathcal{B}}_{(1^k)}$ is in bijection with $\mathcal{A}_k$.

To sketch a proof of Theorem \ref{thm:geneem},	we must define some terms in a manner analogous to Section \ref{subsec:embed}.
	\begin{definition}
	\label{def:ordarr}
 	A set partition array $\mathbf{A}=[A_{ij}]_{s\times n}$ of $\lambda=(\lambda_1,\dotsc,\lambda_s)$ into $n$ parts is a $s \times n$ array of disjoint subsets of $\{1,\dotsc,|\lambda|\}$ such that  $\sqcup_{j=1}^n A_{ij} = \{\sum_{m=1}^{i-1}\lambda_m+1,\dotsc, \sum_{m=1}^{i}\lambda_m\}$.  The set of partition array of $\lambda$ into $n$ parts is denoted by $\mathcal{A}_{\lambda,n}$.
	\end{definition}

\begin{definition}
\label{def:enc}
We encode an element $e_{i_1}\otimes \dotsb \otimes e_{i_{|\lambda|}} \in (F^n)^{\otimes |\lambda|}$ as the indicator function $1_{\mathbf{A}}$ for a set partition array $\mathbf{A}=[A_{ij}]$ of $\lambda$ into $n$ parts by defining $$A_{ij}=\bigg\{r \in \bigg\{\sum_{m=1}^{i-1}\lambda_m+1,\dotsc, \sum_{m=1}^{i}\lambda_m\bigg\}\mid e_{i_r}=e_j\bigg\}.$$
It is easy to see that $(F^n)^{\otimes |\lambda|}$ is isomorphic to the permutation module $F[\mathcal{A}_{\lambda,n}]$.
\end{definition}
For $P=[p_{ij}]\in M(n,\lambda)$, let $D(P)$ consist of arrays of set partition $[A_{ij}]\in \mathcal{A}_{\lambda,n}$ such that the cardinality of the set $A_{ij}$
is $p_{ij}$.

We define the canonical projection map $\pi_{\lambda}:F[\mathcal{A}_{\lambda,n}]\to F[M(n,\lambda)]$ and the map $\tau_{\lambda}: F[M(n,\lambda)]\to F[\mathcal{A}_{\lambda,n}]$ as follows:
	\begin{align}
		\pi_{\lambda}(1_\mathbf{A})&= 1_{[a_{ij}]}, \text{ where } \mathbf{A}=[A_{ij}]\in\mathcal{A}_{\lambda,n} \text{ and } a_{ij}=|A_{ij}|\\
		\tau_{\lambda}(1_{P})&= \frac1{|D(P)|} \sum_{\mathbf{S} \in D(P)}1_{\mathbf{S}}.    
	\end{align}                                                                                                                                     
For $\mathbf{S}\in D(P)$, one has $\pi_\lambda(1_{\mathbf{S}})=1_{P}$  and hence $\pi_\lambda\circ \tau_\lambda=id_{F[M(n,\lambda)]}$. 

	Utilizing the bijection in Equation~\eqref{eq:bij lambda}, we have the following theorem.
	\begin{theorem}\label{thm:geneem}
		The following map is an injective algebra homomorphism
		\begin{align}\label{al:eta-lambda}
			\tilde{\eta}_{\lambda}:\MP_{\lambda}(\xi)&\to \Pa_{|\lambda|}(\xi)\nonumber\\
			\tilde{\eta}_{\lambda}([\Gamma])&=\frac{1}{\alpha_{d_\Gamma}}\sum_{d\in O_{\Gamma}} x_{d},
		\end{align}
		for $[\Gamma]\in\tilde{\mathcal{B}}_{\lambda}$, where $\alpha_{d_\Gamma}$ is the cardinality of $S_\lambda$-orbit of $d_{\Gamma}^{l}$.
		In particular, for $\lambda=(1^k)$, $\MP_{\lambda}(\xi)\cong \Pa_{k}(\xi)$.
	\end{theorem}

\begin{proof}
The structure constants for both $\MP_{\lambda}(\xi)$ and $\Pa_\lambda(\xi)$ are polynomials in $\xi$ therefore it is sufficient to prove $\tilde{\eta}_{\lambda}$ is an algebra homomorphism when 
$\xi$ is evaluated at any positive integer. For this consider the following map, which is an algebra homomorphism,
\begin{align*}
	\eta_{\lambda}:\End_{S_n}(F[M(n,\lambda)]) &\to \End_{S_n}(F[\mathcal{A}_{\lambda,n}]) \\
	\eta_\lambda(f)&= \tau_\lambda\circ f\circ\pi_\lambda,
\end{align*}
where $f\in \End_{S_n}(F[M(n,\lambda)])$. Now consider the following diagram:
\begin{displaymath}
	\xymatrix{
		\MP_{\lambda}(n) \ar@{->>}[d]\ar@{^{(}->}[r]^{\tilde{\eta}_{\lambda}} & \Pa_{\lambda}(n)\ar@{->>}[d]\\
		\End_{S_n}(F[M(n,\lambda)]) \ar@{^{(}->}[r]^{\eta_{\lambda}} & \End_{S_n}(F[\mathcal{A}_{\lambda,n}])
	}
\end{displaymath}                                                                                                                  where the leftmost and the rightmost vertical maps are as in the part $(a)$ of Theorem~\ref{thm:swd} and Theorem~\ref{thm:SWD_partition} respectively. By Schur--Weyl duality the vertical maps are algebra isomorphisms for $n\geqslant 2|\lambda|$. Therefore it is enough to prove 
\begin{equation}\label{eq:eta-lambdag}
	\eta_{\lambda}(T_{[\Gamma]})=\frac{1}{\alpha_{d_\Gamma}}\sum_{d\in O_{\Gamma}} x_{d},  \text{ for } [\Gamma]\in\tilde{\mathcal B}_{\lambda,n}.   
\end{equation}
Equation \eqref{eq:eta-lambdag}  ensures that the  above diagram commutes.  In the following, we define a few more required terms for this set up.

Let $\mathbf{A}=[A_{ij}]\in\mathcal{A}_{\lambda,n}$ and let $P=[p_{ij}]\in M(n,\lambda)$. Define $\set(P)$ to be the set of the columns of $P$ and let $\set(\mathbf{A})=\{\sqcup_{i=1}^sA_{ij}\mid 1\leqslant j\leqslant n\}$. For $d\in\mathcal{A}_{|\lambda|}$, we have
\begin{align*}
	x_{d}(1_{\mathbf{A}})=\begin{cases}
			\sum_{\{\mathbf{S}=[S_{ij}]\mid\{(\sqcup_{i=1}^sA_{ij},\sqcup_{i=1}^sS_{ij})|  1\leqslant j\leqslant n\}=d\}} 1_{\mathbf{S}} & \text{ if } \set(\mathbf{A})=d^u,\\
			0 & \text { otherwise. }
	\end{cases}
\end{align*}
Now the proof of Equation \eqref{eq:eta-lambdag} is exactly the same as Equation \eqref{eq:eta-lambda} so we skip it here.
\end{proof}

The following definition is an analog of $\mathcal{Y}_k$ in this set up and it is used to define an idempotent in $\MP_{\lambda}(\xi)$ which in turn describes the image of embedding in Theorem \ref{thm:geneem}.
\begin{definition}
 Let $\mathcal{Y}_{\lambda}$ be the set of $d= \{(B_{1}^{u},B^{l}_{1}),\ldots,(B_{n}^{u},B_{n}^l)\}$ in $\mathcal{A}_{|\lambda|}$ satisfying the following property, for $1\leqslant i\leqslant s$ and $1\leqslant j\leqslant n$,
	\begin{align*}
		\bigg|B_{j}^{u}\cap \{\sum_{r=0}^{i-1}\lambda_r +1,\dotsc,\sum_{r=0}^{i}\lambda_r\}\bigg|&= \bigg|B_{j}^{l}\cap \{\sum_{r=0}^{i-1}(\lambda_r +1)',\dotsc,\sum_{r=0}^{i}\lambda_i'\}\bigg|.
	\end{align*}
\end{definition}

Recall that the set $\mathcal{Y}_k$ comprises diagrams which have the same number of unprimed elements and primed elements in each component of the diagram. The subset $\mathcal{Y}_\lambda$ comprises partitions whose every component comprises the same number of primed and unprimed numbers from each of the sets $\{1,\dotsc,\lambda_1\},\{\lambda_1+1,\dotsc,\lambda_1+\lambda_2\},\dotsc,\{\mid \lambda\mid +1-\lambda_s,\dotsc, \mid \lambda \mid\}$.

\begin{example} For $\lambda=(2,1)$, we illustrate the partition diagrams appearing in $\mathcal{Y}_\lambda$
	\begin{center}
		\begin{tikzpicture}[scale=0.7,
		mycirc/.style={circle,fill=black, minimum size=0.1mm, inner sep = 1.5pt}]
		\node[mycirc,label=above:{$1$}] (n1) at (0,2) {};
		\node[mycirc,label=above:{$2$}] (n2) at (2,2) {};
		\node[mycirc,label=above:{$3$}] (n3) at (4,2) {};
		\node[mycirc,label=below:{$1'$}] (n1') at (0,1) {};
		\node[mycirc,label=below:{$2'$}] (n2') at (2,1) {};
		\node[mycirc,label=below:{$3'$}] (n3') at (4,1) {};
		\node (n4) at (4.6,1.5) {$, \quad$};
		\draw (n1)--(n1');
		\draw (n2)--(n2');
		\draw (n3)--(n3');
	\end{tikzpicture} 
\begin{tikzpicture}[scale=0.7,
	mycirc/.style={circle,fill=black, minimum size=0.1mm, inner sep = 1.5pt}]
	\node[mycirc,label=above:{$1$}] (n1) at (0,2) {};
	\node[mycirc,label=above:{$2$}] (n2) at (2,2) {};
	\node[mycirc,label=above:{$3$}] (n3) at (4,2) {};
	\node[mycirc,label=below:{$1'$}] (n1') at (0,1) {};
	\node[mycirc,label=below:{$2'$}] (n2') at (2,1) {};
	\node[mycirc,label=below:{$3'$}] (n3') at (4,1) {};
	\node (n4) at (4.6,1.5) {$,\quad$};
	\draw (n1)--(n1');
	\draw (n2)..controls(3,1.5)..(n3);
	\draw (n2')..controls(3,1.5)..(n3');
	\draw (n3)--(n3');
\end{tikzpicture}
\end{center}

\begin{center}
\begin{tikzpicture}[scale=0.7,
	mycirc/.style={circle,fill=black, minimum size=0.1mm, inner sep = 1.5pt}]
	\node[mycirc,label=above:{$1$}] (n1) at (0,2) {};
	\node[mycirc,label=above:{$2$}] (n2) at (2,2) {};
	\node[mycirc,label=above:{$3$}] (n3) at (4,2) {};
	\node[mycirc,label=below:{$1'$}] (n1') at (0,1) {};
	\node[mycirc,label=below:{$2'$}] (n2') at (2,1) {};
	\node[mycirc,label=below:{$3'$}] (n3') at (4,1) {};
	\node (n4) at (4.2,1.5) {$,$};
	\draw (n1)--(n2');
	\draw (n1')--(n2);
		\draw (n3)--(n3');
\end{tikzpicture}
\begin{tikzpicture}[scale=0.7,
	mycirc/.style={circle,fill=black, minimum size=0.1mm, inner sep = 1.5pt}]
	\node[mycirc,label=above:{$1$}] (n1) at (0,2) {};
	\node[mycirc,label=above:{$2$}] (n2) at (2,2) {};
	\node[mycirc,label=above:{$3$}] (n3) at (4,2) {};
	\node[mycirc,label=below:{$1'$}] (n1') at (0,1) {};
	\node[mycirc,label=below:{$2'$}] (n2') at (2,1) {};
	\node[mycirc,label=below:{$3'$}] (n3') at (4,1) {};
	\node (n4) at (4.4,1.5) {$,\quad$};
	\draw (n1)--(n1');
	\draw (n1)..controls(1,1.5)..(n2);
	\draw (n1')..controls(1,1.5)..(n2');
	\draw (n2')..controls(3,1.5)..(n3');
	\draw (n2)..controls(3,1.5)..(n3);
\end{tikzpicture}
\begin{tikzpicture}[scale=0.7,
mycirc/.style={circle,fill=black, minimum size=0.1mm, inner sep = 1.5pt}]
\node[mycirc,label=above:{$1$}] (n1) at (0,2) {};
\node[mycirc,label=above:{$2$}] (n2) at (2,2) {};
\node[mycirc,label=above:{$3$}] (n3) at (4,2) {};
\node[mycirc,label=below:{$1'$}] (n1') at (0,1) {};
\node[mycirc,label=below:{$2'$}] (n2') at (2,1) {};
\node[mycirc,label=below:{$3'$}] (n3') at (4,1) {};
\node (n4) at (4.4,1.5) {$,\quad$};
\draw (n1)..controls(1,1.5)..(n2);
\draw (n1')..controls(1,1.5)..(n2');
\draw  (n1)--(n1'); 
\draw (n3)--(n3');
\end{tikzpicture}
\end{center}

\begin{center}
\begin{tikzpicture}[scale=0.7,
	mycirc/.style={circle,fill=black, minimum size=0.1mm, inner sep = 1.5pt}]
	\node[mycirc,label=above:{$1$}] (n1) at (0,2) {};
	\node[mycirc,label=above:{$2$}] (n2) at (2,2) {};
	\node[mycirc,label=above:{$3$}] (n3) at (4,2) {};
	\node[mycirc,label=below:{$1'$}] (n1') at (0,1) {};
	\node[mycirc,label=below:{$2'$}] (n2') at (2,1) {};
	\node[mycirc,label=below:{$3'$}] (n3') at (4,1) {};
	\node (n4) at (4.4,1.5) {$,\quad$};
	\draw (n2)..controls(3,1.5)..(n3);
	\draw (n1')..controls(3,1.5)..(n3');
	\draw  (n1)--(n2'); 
	\draw (n3)--(n3');
\end{tikzpicture}
\begin{tikzpicture}[scale=0.7,
	mycirc/.style={circle,fill=black, minimum size=0.1mm, inner sep = 1.5pt}]
	\node[mycirc,label=above:{$1$}] (n1) at (0,2) {};
	\node[mycirc,label=above:{$2$}] (n2) at (2,2) {};
	\node[mycirc,label=above:{$3$}] (n3) at (4,2) {};
	\node[mycirc,label=below:{$1'$}] (n1') at (0,1) {};
	\node[mycirc,label=below:{$2'$}] (n2') at (2,1) {};
	\node[mycirc,label=below:{$3'$}] (n3') at (4,1) {};
	\node (n4) at (4.4,1.5) {$,\quad$};
	\draw (n1)..controls(3,1.5)..(n3);
	\draw (n1')..controls(3,1.5)..(n3');
	\draw  (n2)--(n2'); 
	\draw (n3)--(n3');
\end{tikzpicture}
\begin{tikzpicture}[scale=0.7,
	mycirc/.style={circle,fill=black, minimum size=0.1mm, inner sep = 1.5pt}]
	\node[mycirc,label=above:{$1$}] (n1) at (0,2) {};
	\node[mycirc,label=above:{$2$}] (n2) at (2,2) {};
	\node[mycirc,label=above:{$3$}] (n3) at (4,2) {};
	\node[mycirc,label=below:{$1'$}] (n1') at (0,1) {};
	\node[mycirc,label=below:{$2'$}] (n2') at (2,1) {};
	\node[mycirc,label=below:{$3'$}] (n3') at (4,1) {};
	\node (n4) at (4.2,1.5) {$.$};
	\draw (n1)..controls(3,1.5)..(n3);
	\draw (n2')..controls(3,1.5)..(n3');
	\draw  (n1')--(n2); 
	\draw (n3)--(n3');
\end{tikzpicture}
	\end{center}
	\end{example}

\begin{lemma} \label{lm:imageg}
	Let $e=\sum_{d\in \mathcal{Y}_{\lambda}} \frac{1}{\alpha_d}x_d$. 
	Then $e$ is an idempotent in $\Pa_{|\lambda|}(\xi)$ and 
	$\MP_{\lambda}(\xi)\cong e\Pa_{|\lambda|}(\xi)e$.
\end{lemma}
\begin{proof}
	The idea of the proof is exactly same as the proof of Lemma \ref{lm:idem}. One has to check that  $\tilde{\eta}_\lambda(id)=e$, where $\tilde{\eta}_\lambda$ is the embedding defined in Theorem~\ref{thm:geneem} and  $id=\sum_{[\Gamma]\in\mathcal{U}_\lambda}[\Gamma]$ is the identity element of $\MP_\lambda(\xi)$ (see the paragraph above Theorem \ref{thm:multi_rule}). This will give us that $e$ is an idempotent and that the image of $\tilde{\eta}_{\lambda}$ is contained in $e\Pa_{|\lambda|}(\xi)e$. Now the proof of $e\Pa_{|\lambda|}(\xi)e$ is the full image proceeds as in the proof of Lemma \ref{lm:idem}.
\end{proof}  

Under the embedding in Theorem \ref{thm:geneem}, the anti-involution $\mathfrak{i}$ on $\Pa_{|\lambda|}(\xi)$ translates to the anti-involution $\tilde{\mathfrak{i}}$ on $\MP_{\lambda}(\xi)$. For $[\Gamma]\in\tilde{\mathcal{B}}_\lambda$, $\tilde{\mathfrak{i}}([\Gamma])$	is obtained by reflecting $\Gamma$ along the horizontal axis. As a consequence of Lemma \ref{lm:imageg}, we obtain the generic semisimplicity and the cellularity in the following theorem.
	
\begin{theorem}
The algebra $\MP_{\lambda}(\xi)$ is semisimple over $F[\xi]$.  Furthermore, for $v\in F$, $\MP_{\lambda}(v)$ over $F$ is semisimple when $v$ is not an integer or when $v$ is an integer such that $v\notin\{0,1,\ldots, 2|\lambda|-2\}$. Moreover,  $\MP_{\lambda}(v)$ over $F$ is cellular with respect to the anti-involution $\tilde{i}$.
\end{theorem}

	\section*{Acknowledgements}
	The authors thank Amritanshu Prasad for his consistence guidance and fruitful advice. The authors also thank Nate Harman and Mike Zabrocki for their valuable suggestions on this manuscript. We thank the referee for pointing out a few crucial errors in the manuscript and giving us the opportunity to address them. SS was supported by a national postdoctoral fellowship (PDF/2017/000861) of the Department of Science \& Technology, India.

	\bibliographystyle{abbrvurl}

\begin{thebibliography}{10}
	
	\bibitem{MR1194310}
	K.~Akin.
	\newblock On complexes relating the {J}acobi-{T}rudi identity with the
	{B}ernstein-{G}elfand-{G}elfand resolution. {II}.
	\newblock {\em J. Algebra}, 152(2):417--426, 1992.
	\newblock URL: \url{https://doi.org/10.1016/0021-8693(92)90039-O}.
	
	\bibitem{BH}
	G.~Benkart and T.~Halverson.
	\newblock Partition algebras and the invariant theory of the symmetric group.
	\newblock In {\em Recent trends in algebraic combinatorics}, volume~16 of {\em
		Assoc. Women Math. Ser.}, pages 1--41. Springer, Cham, 2019.
	\newblock \href {https://doi.org/10.1007/978-3-030-05141-9_1}
	{\path{doi:10.1007/978-3-030-05141-9_1}}.
	
	\bibitem{Bloss1}
	M.~Bloss.
	\newblock G-colored partition algebras as centralizer algebras of wreath
	products.
	\newblock {\em J. Algebra}, 265(2):690 -- 710, 2003.
	\newblock URL:
	\url{http://www.sciencedirect.com/science/article/pii/S0021869303001327}.
	
	\bibitem{MR1503378}
	R.~Brauer.
	\newblock On algebras which are connected with the semisimple continuous
	groups.
	\newblock {\em Ann. of Math. (2)}, 38(4):857--872, 1937.
	\newblock URL: \url{https://doi.org/10.2307/1968843}.
	
	\bibitem{OZS}
	L.~Colmenarejo, R.~Orellana, F.~Saliola, A.~Schilling, and M.~Zabrocki.
	\newblock An insertion algorithm on multiset partitions with applications to
	diagram algebras.
	\newblock {\em J. Algebra}, 557:97--128, 2020.
	\newblock URL: \url{https://doi.org/10.1016/j.jalgebra.2020.04.010}.
	
	\bibitem{GL}
	J.~J. Graham and G.~I. Lehrer.
	\newblock Cellular algebras.
	\newblock {\em Invent. Math.}, 123(1):1--34, 1996.
	\newblock URL: \url{https://doi.org/10.1007/BF01232365}.
	
	\bibitem{HR}
	T.~Halverson and A.~Ram.
	\newblock Partition algebras.
	\newblock {\em European J. Combin.}, 26(6):869--921, 2005.
	\newblock URL: \url{http://dx.doi.org/10.1016/j.ejc.2004.06.005}.
	
	\bibitem{Nate}
	N.~Harman.
	\newblock Representations of monomial matrices and restriction from {$GL_n$} to
	{$S_n$}.
	\newblock \href{https://arxiv.org/abs/1804.04702}{arXiv:1804.04702}, 2018.
	
	\bibitem{Jones}
	V.~F.~R. Jones.
	\newblock The {P}otts model and the symmetric group.
	\newblock In {\em Subfactors ({K}yuzeso, 1993)}, pages 259--267. World Sci.
	Publ., River Edge, NJ, 1994.
	
	\bibitem{KX}
	S.~K\"{o}nig and C.~Xi.
	\newblock On the structure of cellular algebras.
	\newblock In {\em Algebras and modules, {II} ({G}eiranger, 1996)}, volume~24 of
	{\em CMS Conf. Proc.}, pages 365--386. Amer. Math. Soc., Providence, RI,
	1998.
	
	\bibitem{littlewood_1958}
	D.~E. Littlewood.
	\newblock Products and plethysms of characters with orthogonal, symplectic and
	symmetric groups.
	\newblock {\em Canad. J. Math.}, 10:17–32, 1958.
	\newblock \href {https://doi.org/10.4153/CJM-1958-002-7}
	{\path{doi:10.4153/CJM-1958-002-7}}.
	
	\bibitem{MR3443860}
	I.~G. Macdonald.
	\newblock {\em Symmetric functions and {H}all polynomials}.
	\newblock Oxford Classic Texts in the Physical Sciences. The Clarendon Press,
	Oxford University Press, New York, second edition, 2015.
	
	\bibitem{Martin}
	P.~Martin.
	\newblock {\em Potts models and related problems in statistical mechanics},
	volume~5 of {\em Series on Advances in Statistical Mechanics}.
	\newblock World Scientific Publishing Co., Inc., Teaneck, NJ, 1991.
	\newblock URL: \url{https://doi.org/10.1142/0983}.
	
	\bibitem{MR1265453}
	P.~Martin.
	\newblock Temperley-{L}ieb algebras for nonplanar statistical mechanics---the
	partition algebra construction.
	\newblock {\em J. Knot Theory Ramifications}, 3(1):51--82, 1994.
	\newblock URL: \url{https://doi.org/10.1142/S0218216594000071}.
	
	\bibitem{martin2008diagram}
	P.~Martin.
	\newblock On diagram categories, representation theory and.
	\newblock In {\em Noncommutative Rings, Group Rings, Diagram Algebras, and
		Their Applications: International Conference, December 18-22, 2006,
		University of Madras, Chennai, India}, volume 456, page~99. American
	Mathematical Soc., 2008.
	
	\bibitem{MS94}
	P.~Martin and H.~Saleur.
	\newblock Algebras in higher-dimensional statistical mechanics---the
	exceptional partition (mean field) algebras.
	\newblock {\em Lett. Math. Phys.}, 30(3):179--185, 1994.
	\newblock URL: \url{https://doi.org/10.1007/BF00805850}.
	
	\bibitem{MS}
	A.~Mishra and S.~Srivastava.
	\newblock On representation theory of partition algebras for complex reflection
	groups.
	\newblock {\em Algebr. Comb.}, 3(2):389--432, 2020.
	\newblock URL: \url{alco.centre-mersenne.org/item/ALCO_2020__3_2_389_0/}.
	
	\bibitem{OZ}
	R.~Orellana and M.~Zabrocki.
	\newblock Symmetric group characters as symmetric functions.
	\newblock {\em Adv. Math.}, 390:107943, 2021.
	\newblock URL: \url{https://doi.org/10.1016/j.aim.2021.107943}.
	
	\bibitem{OZM}
	R.~{Orellana} and M.~{Zabrocki}.
	\newblock {Howe duality of the symmetric group and a multiset partition
		algebra}.
	\newblock {\em To appear in Comm. Algebra}, 2022.
	
	\bibitem{rtcv}
	A.~Prasad.
	\newblock {\em Representation Theory: A Combinatorial Viewpoint}.
	\newblock Number 147 in Cambridge Studies in Advanced Mathematics. Cambridge
	University Press, Delhi, 2015.
	
	\bibitem{Schur1927}
	I.~Schur.
	\newblock {\"U}ber die rationalen {D}arstellungen der allgemeinen linearen
	{G}ruppe.
	\newblock {\em S'ber Akad. Wiss. Berlin}, pages 100--124, 1927.
	
	\bibitem{MR1676282}
	R.~P. Stanley.
	\newblock {\em Enumerative combinatorics. {V}ol. 2}, volume~62 of {\em
		Cambridge Studies in Advanced Mathematics}.
	\newblock Cambridge University Press, Cambridge, 1999.
	\newblock URL: \url{https://doi.org/10.1017/CBO9780511609589}.
	
	\bibitem{Weyl}
	H.~Weyl.
	\newblock {\em The Classical Groups: Their Invariants and Representations}.
	\newblock Princeton University Press, Princeton, NJ, 1939.
	
	\bibitem{Xi}
	C.~Xi.
	\newblock Partition algebras are cellular.
	\newblock {\em Compositio Math.}, 119(1):99--109, 1999.
	\newblock URL: \url{https://doi.org/10.1023/A:1001776125173}.
	
	\bibitem{Young}
	A.~Young.
	\newblock On {Q}uantitative {S}ubstitutional {A}nalysis ({S}econd {P}aper).
	\newblock {\em Proc. Lond. Math. Soc.}, 34:361--397, 1902.
	\newblock URL: \url{https://doi.org/10.1112/plms/s1-34.1.361}.
	
\end{thebibliography}

\end{document}